\documentclass[a4paper,11pt]{article}
\usepackage[T1]{fontenc}
\usepackage[latin1]{inputenc}
\usepackage{yfonts}
\usepackage{geometry}
\usepackage[english]{babel}
\usepackage{amsthm,amssymb,amsfonts,amsmath,color,graphicx,float,url,bbm,caption}
\usepackage[authoryear,round,sort]{natbib}
\usepackage{geometry}
\DeclareMathAlphabet{\mathpzc}{OT1}{pzc}{m}{it}

\newcommand{\R}{\mathbb R}
\newcommand{\E}{\mathbb E}
\newcommand{\e}{\varepsilon}
\newcommand{\I}{\mathbf 1}

\DeclareMathOperator*{\var}{Var}

\newtheorem{rem}{Remark}[section]

\newtheorem{lem}{Lemma}[section]

\newtheorem{prop}{Proposition}[section]
\newtheorem{theo}{Theorem}[section]

\begin{document}

\title{ Improved rates for Wasserstein deconvolution with ordinary smooth error
 in dimension one}

\author{J\'er\^ome Dedecker$^{(1)}$, Aur\'elie Fischer$^{(2)}$ and Bertrand Michel$^{(3)}$}

\maketitle

{ \scriptsize
\noindent (1)  Laboratoire MAP5  UMR  CNRS 8145,  Universit\'e Paris  Descartes,
Sorbonne Paris Cit\'e \\
(2) LPMA UMR CNRS  7599, Universit\'e Paris Diderot, Sorbonne Paris Cit\'e  \\
(3) LSTA, Universit\'e Pierre et Marie Curie}

\begin{abstract}
This paper deals with the estimation of a probability measure on the real line  from data observed with an additive noise. We are interested in  rates of convergence for the Wasserstein metric of order $p\geq 1$. The distribution of the errors is assumed to be known and to belong to a class of supersmooth or ordinary smooth distributions. We obtain in the univariate situation an improved upper bound in the ordinary smooth case and less restrictive conditions for the existing bound in the supersmooth one.
In the ordinary smooth case, a lower bound  is also provided, and numerical experiments illustrating the rates of convergence are presented.

\end{abstract}




\section{Introduction}

Consider the following convolution model: we observe $n$ real-valued random variables  $Y_1, \ldots , Y_n$
such that
\begin{equation}\label{ModeConvo}Y_i=X_i+\e_i,\end{equation} where the $X_i$'s are independent and identically distributed  according to an unknown probability $\mu$, which we want to estimate. The random variables
$\e_i$, $i=1,\dots,n$, are independent and identically distributed according to a known probability measure $\mu_\e$,
not necessarily symmetric. Moreover we assume that
 $(X_1, \ldots , X_n)$ is independent of  $(\e_1, \ldots, \e_n)$.

The purpose of the paper is to investigate rates of convergence for the estimation of the measure $\mu$ under Wasserstein metrics.
For $p\in[1,\infty)$, the Wasserstein distance $W_p$ between $\mu$ and $\nu$ is given by $$W_p(\mu,\nu)=\inf _{\pi \in \Pi(\mu,\nu)}\left(\int_{\R^2}|x-y|^p\pi(dx,dy)\right)^{\frac 1 p},$$
where $\Pi(\mu,\nu)$ is the set of probability measures on ${\mathbb R} \times {\mathbb R}$ with marginal distributions $\mu$ and $\nu$ (see \cite{RR98} or \cite{Vil08}).
The distances $W_p$  are natural metrics for comparing measures. For instance they can  compare two singular measures, which is of course impossible with the functional metrics commonly used in density estimation. Convergence of measure under Wasserstein distances is an  active domain of research in probability and statistics. For instance, the rate of convergence of the empirical measure under these metrics has been obtained recently by both \cite{DSS13} and \cite{FG} in $\R^d$ and also by \cite{bobkov2014one} in the one-dimensional framework. Moreover, Wasserstein metrics are involved in many fields of mathematics and computer sciences. For instance, in the field of Topological Data Analysis (TDA) \citep{carlsson2009topology}, Wasserstein distances recently appeared to be natural metrics  for controlling the estimation of geometric and topological features of the sampling measure and its support. Indeed, in \cite{CCSM11}, a distance function to measures is introduced to solve geometric inference problems in a probabilistic setting: if a known measure $\nu$ is close enough with respect to $W_2$ to a measure $\mu$ concentrated on a given shape, then the topological properties of the shape can be recovered by using the distance to $\nu$. More generally,  the Wasserstein loss could be used as a guide for inferring the support (see the Cantor experiment in Section~\ref{exp:Cantor}). Other results  in TDA with stability results involving the Wasserstein distances can be found in \cite{guibas2013witnessed} and \cite{chazal2014subsampling}.  In practice, the data can be observed with noise, which motivates in this framework the study of the Wasserstein deconvolution problem \citep{CCDM11}, in particular if the deconvolved measure and  the ``true measure''  are singular.

Rates of convergence in deconvolution have mostly been considered in density estimation,
for pointwise or global convergence.
Minimax rates can be found for instance in \cite{Fan91}, \cite{BT08a},
\cite{BT08b} and in the monograph of \cite{Mei09}.
In this paper, however, we  shall not assume that $\mu$ has a density with respect to the Lebesgue measure.
In this context, rates of convergence for the $W_2$ Wasserstein distance have first been studied for several noise distributions by \cite{CCDM11}.
Recently, \cite{DM13} have obtained optimal rates of convergence in the minimax sense for a class of supersmooth error distributions, in any dimension, under any Wasserstein metric $W_p$.
The result relies on the fact that lower bounds in any dimension can be deduced in this case from the lower  bounds in dimension 1. Such a method cannot be used in the ordinary smooth case, where the rate of convergence depends on the dimension.
As noticed by \cite{Fan91}, establishing optimal rates of convergence in the ordinary smooth case is more difficult than in the supersmooth one, even for pointwise estimation.

A key fact in the univariate context is that Wasserstein metrics are linked to
 integrated risks between cumulative distribution functions (cdf), see
the upper bound (\ref{eq:Rio}) below.
In dimension 1, when estimating  the density of $\mu$, optimal rates of convergence for
integrated risks  can be found in \cite{Fan91a, Fan93}. When estimating the cdf $F$
of $\mu$, optimal rates for the
pointwise and integrated quadratic risks are given in \cite{HL}, where
it is shown in particular that the  rate $\sqrt n$ can be reached when the error distribution is ordinary smooth with a smoothness index less than $1/2$.
Concerning the pointwise estimation
of $F(x_0)$, optimal rates for the
quadratic risk are also given in \cite{DGJ11}, when the density of $\mu$  belongs
to a Sobolev class.

The case $\beta=0$ in the upper bound (3.9) of \cite{HL} corresponds to the case where
no assumption (except a moment assumption) is made on the measure $\mu$ (in particular $\mu$ is not assumed to be absolutely continuous with respect to the Lebesgue measure). This is precisely the case which we want to consider in the present paper. However the results by \cite{HL}
cannot be applied to the Wasserstein deconvolution problems for two reasons:
firstly, the integrated quadratic risk for estimating a cdf is not linked to
Wasserstein distances, and secondly, the estimator of the cdf of $\mu$ proposed
in \cite{HL} is  the cdf of a signed measure, and is not well defined
as an estimator of $\mu$ for the Wasserstein metric.

In the present contribution, we propose in the univariate situation an improved upper bound for deconvolving $\mu$ under $W_p$, and a lower bound when the error is ordinary smooth. We  recover the optimal rate of convergence in the supersmooth case with slightly weaker regularity conditions than in \cite{DM13}. The estimator
of the cdf $F$ of $\mu$ is built in two steps:
firstly, as in \cite{HL}, we define a preliminary estimator through a
classical kernel deconvolution method, and secondly we take an appropriate isotone
approximation of this estimator. For controlling  the random term, we use a moment inequality on the cdfs, which is due to \cite{E71}. To be complete, we show in Section~\ref{sec:lowerbound} that for $p>1$, the Wasserstein deconvolution problem is different from the cdf deconvolution problem with loss $L_p$ associated to 
\`Ebralidze's inequality (see (\ref{loss}) for the definition).

The paper is organized as follows. In Section \ref{sec:withouterror},
some facts about the case without error are recalled and discussed. The upper bounds for
Wasserstein  deconvolution  with supersmooth or ordinary smooth errors are given in Section~\ref{sec:upperbound}, and Section~\ref{sec:lowerbound} is about lower bounds. The upper bound is proved in Section~\ref{sec:proof}. Section \ref{sec:simu} presents the implementation of the method and some experimental results.
In particular, observed rates of convergence are compared with the theoretical bounds for the Wasserstein metrics $W_1$ and $W_2$, and we study as an illustrative example the deconvolution of the  uniform measure on the Cantor set.

\section{On the case without error}\label{sec:withouterror}

We begin by considering the simple case when one observes directly  $X_1, \dots, X_n$ with values in $\R$ without error.
Let us recall some results for the quantities $W_p(\mu_n,\mu)$, where $\mu_n$ is the empirical measure, given by $$\mu_n=\frac 1 n \sum_{i=1}^n\delta_{X_i}.$$
Let $F$ be the cdf of $X_1$, $F_n$  the cdf of $\mu_n$, and let $F^{-1}$ and $F_n^{-1}$ be their usual cadlag inverses. Recall that, for any $p \geq 1$,
\begin{equation} \label{WpFinv}
W_p^p(\mu_n,\mu)=\int_0^1|F_n^{-1}(u)-F^{-1}(u)|^p du \, ,
\end{equation}
and if $p=1$:
$$
W_1(\mu_n,\mu)=\int|F_n ^{-1} (u)-F^{-1}(u)| du   = \int|F_n(t)-F(t)| dt  \,  .
$$

The case $p=1$ has been well understood since the paper by \cite{BGM99}. The random variable $\sqrt{n}W_1(\mu_n,\mu)$ converges in distribution to $\int |B(F(t))|dt$, where $B$ is a standard Brownian bridge, if and only if \begin{equation}
\int_0^\infty\sqrt{P(|X|>x)}dx<\infty\label{eq:momfaible} ,
\end{equation}
or equivalently if 
\begin{equation*}
\int_0^\infty\sqrt{ F(x) (1-F(x)) }dx<\infty .
\end{equation*}
More recently, \cite{bobkov2014one} have shown that the rate of  $\E W_1(\mu_n,\mu)$ can be characterized by the quantities
$$ \int_{4 n F(x) (1-F(x))\leq 1 }  F(x) (1-F(x) )  dx  \quad \textrm{ and } \quad  \int_{4 n F(x) (1-F(x)) > 1 }   \sqrt { F(x) (1-F(x) ) }   dx  . $$ 
More precisely,  the rate $1 / \sqrt n $ is achieved if and only if \eqref{eq:momfaible} is satisfied. When this is not the case, $\E W_1(\mu_n,\mu)$ may decay at an arbitrary slow rate. See the Theorems 3.3 and 3.5 in their paper.

For $p>1$, the situation is more complicated. Extra conditions are necessary  to ensure that $W_p(\mu_n,\mu)$ is of order $1/ \sqrt n$. If the random variables
take their values in a compact interval $[a,b]$ and if the cdf $F$ is continuously
differentiable on $[a,b]$ with strictly positive derivative $f$, then $n^{p/2}W_p^p(\mu_n,\mu)$ converges in distribution to
$\int_0^1 |B(u)|^p/|f\circ F^{-1}(u)|^p du$ (see Lemma 3.9.23 in
\cite{vdVW96}).  But in general, the rate can be much slower.
The convergence in distribution for the case $p=2$ has been studied in detail by \cite{BGU05}.
Under additional conditions on $F$ (see condition (2.7) in \cite{BGU05}, which requires in particular that $F$ is twice differentiable), the rate of convergence depends on the behavior of $F^{-1}$ in a neighborhood of 0 and 1.
For instance, if $$F(t)=\left(1-\frac 1 {t^{\alpha-1}}\right)\I_{\{ t>1\}},$$ where $\alpha>3$, it follows from Theorem 4.7 in \cite{BGU05} that \begin{equation}n^{(\alpha-3)/(\alpha-1)}W_2^2(\mu_n,\mu)\label{eq:ex}\end{equation} converges in distribution. The limiting distribution is explicitly given in \cite{BGU05}. 

The rates of decay of $\E W_p(\mu_n,\mu)$ and  $[ \E W_p ^p(\mu_n,\mu)]^{1/p}$ have been studied more recently in \cite{bobkov2014one}. They show that these quantities decay at the standard rate $1/\sqrt n$ if and only if 
$$
J_p (\mu) = \int_{ \R}   \frac{ \left[ F(x) (1-F(x) ) \right]^{p/2} dx}{f(x) ^{p-1}}  < \infty \, ,
 $$ 
where $f$ is the density of the absolutely continuous component of $\mu$. In particular (see their Theorem 5.6), they show that 
 $$[ \E W_p ^p(\mu_n,\mu)]^{1/p} \leq  \frac {5p}{\sqrt{n+2}} J_p ^{1/p} (\mu) . $$
However, this approach cannot be applied when the measure $\mu$ and the Lebesgue measure are singular.
An alternative approach to obtain the rate of decay of $ \E W_p ^p(\mu_n,\mu) $   is to use the following inequality,
due to \cite{E71} (see also Sections 7.4 and 7.5 in \cite{bobkov2014one}) :
for any $p \geq 1$,
\begin{equation}
W_p^p(\mu, \nu)  \leq \kappa_p \int |x|^{p-1} |F_\mu-F_\nu|(x) dx \, ,
\label{eq:Rio}\end{equation}
where $\kappa_p= 2^{p-1} p$.
Starting from (\ref{eq:Rio}), we get that
\begin{align*}
\E W_p^p(\mu_n,\mu) & \leq  \int  |x|^{p-1}  \E | F_n(x) - F(x) |  dx \\
& \leq  \int  |x|^{p-1} \sqrt{  \E | F_n(x) - F(x) |^2 } dx \\
& \leq  \frac 1 {\sqrt n}  \int  |x|^{p-1}   \sqrt{ F(x) (1-F(x))}  dx
\end{align*}
where $F_n$ is the empirical cdf. Now, it is easy to see  that  this last integral is finite if and only
if
\begin{equation}\label{eq:tailcond}
\int_0^\infty |x|^{p-1}\sqrt{P(|X|>x)}dx<\infty \, .
\end{equation}
It follows that $\E W_p^p(\mu_n,\mu)\leq Cn^{-1/2}$
as soon as \eqref{eq:tailcond} is satisfied. For instance, taking $p=2$, a tail satisfying $P(|X|>x)=\mathcal O \Big(\frac 1 {x^4\log x^{2+\e}}\Big)$ gives the rate $1/ \sqrt n$. Hence, we obtain the same rate as
in (\ref{eq:ex}) for $\alpha=5$, with a slightly stronger tail condition (due to the fact that we control the expectation), but without additional assumptions on the cdf $F$. 

Since we want to estimate singular measures, we shall follow this approach in the sequel.

\section{Upper bounds for $W_p$ in deconvolution}\label{sec:upperbound}

\subsection{Construction of the estimator}

Let us start with some notations.
For $\mu$ a probability measure and $\nu$ another probability measure, with density $g$, we denote by $\mu\star g$ the density of $\mu\star \nu$, given by
$$\mu\star g(x)=\int_{\R}g(x-y)\mu(dy).$$
We further denote by $\mu^*$ (respectively $f^*$) the Fourier transform of the probability
measure $\mu$ (respectively of the integrable function $f$), that is $$\mu^*(x)=\int_\R  e^{iux}\mu(du) \quad  \mbox{and}
\quad  f^*(x)=\int_\R  e^{iux}f(u) du.$$
Finally, let $F$ be  the cumulative distribution function of $\mu$.

\medskip

The estimator $\tilde \mu_n$ of the measure $\mu$ is built in two steps:


\begin{enumerate}
\item {\bf A preliminary estimator of $F$.} Let $\lceil p\rceil $ be the least integer greater than or equal to $p$. We first introduce a symmetric nonnegative kernel $k$ such that its Fourier transform $k^*$ is $\lceil p\rceil $ times differentiable with Lipschitz $\lceil p \rceil-$th derivative and is supported on $[-1,1]$.
An example of  such a kernel is given by \begin{equation}\label{eq:kernelp}k(x)=C_p\left[\frac{(2\lceil p/2\rceil+2)\sin \frac{x}{2\lceil p/2\rceil+2}}x\right]^{2\lceil p/2\rceil+2},\end{equation} where $C_p$ is such that $\int k(x)dx=1$.

 We  define now a preliminary estimator $\hat F_n$ of $F$:
\begin{equation} \label{eq:naive-estim}
\hat{F_n}(t)=\frac 1 {nh}\int_{-\infty}^t\sum_{k=1}^n\tilde{k}_h\left(\frac{u-Y_k}{h}\right)du
\end{equation}
 where
$$
 \tilde{k}_h(x)=\frac 1 {2\pi}\int\frac {e^{iux }k^*(u)}{\mu_\e^*(-u/h)}du.
$$
Let us first give some conditions under which these quantities are well defined. Clearly, $\tilde{k}_h(x)$ is well defined as soon as $\mu_\e^*$
does not vanish, since in that case it is the Fourier transform of a continuous and compactly supported function (it can be easily checked that $\tilde{k}_h(x)$ is a real function). In the sequel, we shall
always assume that $r_\e= 1/ \mu_\e^*$ is at least two times continuously differentiable. In that case,
the function $w(u)=\frac {k^*(u)}{\mu_\e^*(-u/h)}$ is two times differentiable with bounded and compactly supported derivatives.
An integration by parts yields
$$
   \tilde{k}_h(x)=- \frac 1 {2\pi i x}\int e^{iux } w'(u) du \quad \text{and} \quad
   \tilde{k}_h(x)=- \frac 1 {2\pi x^2}\int e^{iux } w''(u) du.
$$
It follows that $\tilde{k}_h$ is a continuous function such that $\tilde{k}_h(x)={\mathcal O}(1/(1+x^2))$. Hence  $\tilde{k}_h$ belongs to
${\mathbb L}^1(dx)$ and $\hat F_n$ is well defined. Now the inverse Fourier formula gives that
$
 \tilde{k}^*_h(x)= \frac{k^*(u)}{\mu_\e^*(u/h)} 
$.
Consequently $\tilde{k}^*_h(0)=1$, proving that $\int \tilde{k}_h(x) dx=1$ and that $\lim_{t \rightarrow \infty} \hat F_n(t)=1$.

  However, this estimator $\hat{F_n}$, based on the standard deconvolution kernel density estimator $\tilde{k}_h$ first introduced by \cite{CH88}, is not a cumulative distribution function since it is not necessarily non-decreasing.

\item {\bf Isotone approximation.}
We need to define an estimator $\tilde{F_n}$ of $F$ which is a cumulative distribution function.  We choose the estimator $\tilde{F_n}$ as an approximate minimizer over all distribution functions of the quantity $\int_\R|x|^{p-1}|\hat{F}_n-G|(x)dx$. More precisely,
given $\rho >0$, let $\tilde F_n$ be such that, for every distribution function $G$,
\begin{equation*}\label{def:tildeF}
\int |x|^{p-1}|\hat{F}_n-\tilde{F_n}|(x)dx\leq \int |x|^{p-1}|\hat{F}_n-G|(x)dx+\rho \, .
\end{equation*}
Here, $\rho$ may be chosen equal to $0$ (best isotone approximation) but  the condition
$\rho= O(n^{-1/2})$ is the only condition required to get the rates of Section
\ref{MR} below.

The estimator $\tilde \mu_n$ is then defined by:
\begin{equation}
\label{eq:defest}
\text{$\tilde \mu_n$ is the probability measure with distribution function $\tilde F_n$.}
\end{equation}
\end{enumerate}\begin{rem}
The second step is different from that of \cite{DM13}, who choose $\tilde \mu_n$ as the (normalized)
positive part of $\mu_n$. As we shall see, the isotone approximation allows to get better rates of convergence
in the ordinary smooth case. The superiority of the isotone estimator will also be clearly
highlighted through the simulations (see Section \ref{sec:simrates}). However, this approach
works only in the one-dimensional case.

  One may argue that the estimator $\tilde{F_n}$ is not explicit, and can be quite difficult
  to compute, because the minimization is done over an infinite dimensional set. In fact, this is
  not an issue, because powerful algorithms have been developed to deal with this situation.
  In Section \ref{sec:simu}, we shall
  use the function {\ttfamily gpava}  from the R package {\ttfamily isotonic} \citep{mair2009isotone}
  (see Section \ref{sec:firststep} for more details).

\end{rem} \subsection{First upper bounds for $W^p_p(\tilde{\mu}_n,\mu)$.}
The control of $W^p_p(\tilde{\mu}_n,\mu)$ is done in three steps:
\begin{enumerate}
\item {\bf A bias/random decomposition.} Denoting by  $K_h$ the function $h^{-1}k( \cdot / h) $, we have that
\begin{equation}\label{eq:start}
 W^p_p(\tilde{\mu}_n,\mu)\leq 2^{p-1}W_p^p(\mu \star K_h,\mu)+2^{p-1} W_p^p(\tilde{\mu}_n,\mu\star K_h).
\end{equation}
The non-random quantity $W_p^p(\mu \star K_h,\mu)$ is the {\it bias} of the estimator $\tilde \mu_n$.

\item {\bf Control of the bias.} Let $V_h$ be a random variable with distribution $K_h$ and
independent of $X_1$, in such a way that the distribution of $X_1 + V_h$ is $\mu\star K_h$.
By definition of $W_p$, we have
\begin{equation}
W^p_p(\mu\star K_h,\mu) \leq
\E |X_1+V_h-X_1|^p = \E|V_h|^p=h^p\int |x|^pk(x)dx.\label{eq:bias}
\end{equation}

\item {\bf Control of the random term.} Note that
$$
\E[\hat{F}_n(t)]= \int_{-\infty}^t \mu\star K_h (x) dx
$$ is the cdf of $\mu \star K_h$.   Applying   \`Ebralidze's inequality
(\ref{eq:Rio}), we obtain that
$$
W_p^p(\tilde{\mu}_n,\mu\star K_h)\leq \kappa_p \int |x|^{p-1}|\tilde{F}_n-\E[\hat{F}_n]|(x)dx \, .
$$
Now, by the triangle inequality and the definition of $\tilde{F}_n$,
\begin{align}
 W_p^p(\tilde{\mu}_n,\mu\star K_h)
&\leq \kappa_p \left(\int |x|^{p-1}|\tilde{F}_n-\hat{F}_n|(x)dx+\int |x|^{p-1}|\hat{F}_n-\E[\hat{F}_n]|(x)dx\right)\nonumber\\
&\leq \rho +2\kappa_p \int|x|^{p-1}|\hat{F}_n-\E[\hat{F}_n]|(x)dx \, .
\label{eq:deb} \end{align}
\end{enumerate}
From (\ref{eq:start}), (\ref{eq:bias}) and (\ref{eq:deb}),
to get explicit rates of convergence for ${\mathbb E}[W^p_p(\tilde{\mu}_n,\mu)]$,
it remains to control the term
$$
{\mathbb E} \left ( \int|x|^{p-1}|\hat{F}_n-\E[\hat{F}_n]|(x)dx  \right )\, .
$$

\begin{rem}
Another main difference between the present paper and \cite{DM13} is the use of
\`Ebralidze's inequality (\ref{eq:Rio}) to control the random term. In \cite{DM13}
the term  $W_p^p(\tilde{\mu}_n,\mu\star K_h)$ (for another choice of $\tilde \mu_n$)
is bounded by a term involving the variation norm between $\tilde \mu_n$ and $\mu$.
In our case, this upper bound would give  a worse rate of convergence.

Note that Inequality (\ref{eq:Rio}) is used here to control the random term only.
A possible alternative approach is to use (\ref{eq:Rio}) directly, as in the case without error
(see Section \ref{sec:withouterror}). This would give the upper bound
\begin{equation}\label{badbound}
W^p_p(\tilde{\mu}_n,\mu) \leq \kappa_p \int |x|^{p-1} |\tilde F_n-F|(x) dx \, .
\end{equation}
 In that case, the bias term would be
$$
  \int |x|^{p-1} |\E[\hat{F}_n]-F|(x) dx \, .
$$
However, without extra regularity assumptions on $\mu$, this would give a bias term of order
$h$, and then the same rate of convergence as in the case $p=1$, that is $n ^{1/(2 \beta + 1)}$
under the assumptions of Theorem  
\ref{theo:upperbounds}  (Item 2)  of the next section.  But this rate is always too slow
for $p>1$, see again  Theorem \ref{theo:upperbounds}. Moreover, there
is no hope to obtain a better rate from (\ref{badbound}) because $n ^{1/(2 \beta + 1)} $ is  also the minimax rates to estimate $F$
with the loss function
\begin{equation}\label{loss}
L_p(G,F)= \int |x|^{p-1} |G-F|(x) dx \, ,
\end{equation}
This last assertion comes from the lower bound stated in Theorem~\ref{Theo:lowercdfLp} of   Section \ref{sec:lowerbound}.
\end{rem}

\subsection{Main results}\label{MR}

\bigskip

Let $r_\e= 1/ \mu_\e^*$, and let
$r_\e^{(\ell)}$ be the $\ell$-th derivative of $r_\e$. Let $m_0$
denote the least integer strictly greater than
$p+\frac 12$, and $m_1$ be the least integer strictly greater than $p-\frac 12$.

Our first result is a general proposition which gives an upper bound for
$\E W_p^p(\tilde{\mu}_n, \mu)$ involving a tail condition on $Y$ and the regularity
of $r_\e$.
\begin{prop}
\label{prop:gen}
Let $\rho\leq n^{-1/2}$, and let $\tilde{\mu}_n$ be the estimator defined in (\ref{eq:defest}). Assume that $r_\e$ is $m_0$ times continuously differentiable. For any $h\leq 1$, we have
$$\E W_p^p(\tilde{\mu}_n, \mu)\leq\frac 1 {\sqrt{n}} +h^p 2^{p-1}\int|x|^pk(x)dx\\+ \frac {C} {\sqrt{n}} (A_1+A_2+ A_3 +A_4)
$$
where
\begin{align*}
A_1 &= \Big( \sup_{t \in [-2,2]}\sum_{\ell=0}^{1}|r_\e^{(\ell)}(t)| \Big)\int_0^\infty |x|^{p-1}\sqrt{P\left(|Y|\geq  x \right)}dx \\
A_2 &= \sup_{t \in [-2,2]}\sum_{\ell=0}^{m_0}|r_\e^{(\ell)}(t)| \\
A_3 &= \left [\E |Y|^{2p-\frac 1 2 }\int_{-1/h}^{1/h}
\frac{|r_\e(x)|^2}{|x|^2}\I_{[-1,1]^c}(x)dx \right]^{1/2}\\
A_4 &=  \left[ \sum_{\ell=0}^{m_1}\int_{-1/h}^{1/h}
\frac{|r_\e^{(\ell)}(x)|^2}{|x|^2}\I_{[-1,1]^c}(x)dx \right ]^{1/2}.
\end{align*}
\end{prop}

For the sake of readability, the proof of Proposition \ref{prop:gen} is postponed to Section \ref{sec:proof}.

We are now in a position to give the rates of convergence for
the Wasserstein deconvolution, for a class of supersmooth error distributions, and for
a class of ordinary smooth error distributions.

\begin{theo} \label{theo:upperbounds}
Let $\rho\leq n^{-1/2}$, and let $\tilde{\mu}_n$ be the estimator defined in (\ref{eq:defest}).
Assume that
 \begin{equation}\int_0^\infty |x|^{p-1}\sqrt{P(|Y|\geq x)}dx<\infty \ \mbox{and} \ \sup_{t\in[-2,2]}|r_\e^{(m_0)}(t)|<\infty \label{eq:hyp}.\end{equation}

\begin{enumerate}
\item Assume that there exist $\beta>0$, $\tilde{\beta}\geq 0$, $\gamma>0$ and $c>0$, such that for every $\ell\in \{0,1,\dots, m_1\}$ and  every $t\in \R$,
\begin{equation}
|r_\e^{(\ell)}(t)|\leq c(1+|t|)^{\tilde{\beta}}\exp(|t|^\beta/\gamma).\label{eq:super}\end{equation}

Then, taking $h= (4/(\gamma\log n ))^{1/\beta}$,
there exists a positive constant C such that $$\E W_p^p(\tilde{\mu}_n, \mu)\leq C (\log n)^{-p/\beta}.$$
\item Assume that there exist $\beta>0$ and $c>0$, such that for every $\ell\in \{0,1,\dots, m_1\}$ and  every $t\in \R$, \begin{equation}|r_\e^{(\ell)}(t)|\leq c(1+|t|)^{\beta}.\label{eq:ordinary}\end{equation}
Then, taking $h= n^{-\frac{1}{2p+(2\beta-1)_+}}$, there  exists a positive constant
C such that
\begin{equation}\label{eq:rateordi}
\E W_p^p(\tilde{\mu}_n, \mu)\leq C \psi_n \, ,
\end{equation} where
$$\psi_n=\begin{cases}
 n^{-\frac p{2p+2\beta  -1}} &\mbox{ if } \beta >\frac 1 2\\
\sqrt{\frac{\log n}n} &\mbox{ if } \beta=\frac 1 2\\
\frac 1 {\sqrt{n}} &\mbox{ if } \beta <\frac 1 2.
\end{cases}$$
\end{enumerate}
\end{theo}
This result requires several comments.
\begin{rem}
In the ordinary smooth case, when $\beta< 1/2$, any bandwidth
$h=\mathcal O  (n^{-1/2p})$ leads to the rate $n^{-1/2}$.
The fact that there are three different
situations according as $\beta>1/2, \beta=1/2$ or $\beta<1/2$ has already been
pointed out in Theorem 3.2 of \cite{HL} and in Theorem 2.1 of  \cite{DGJ11} for the
estimation of the cdf $F$.
Note that the estimator $\hat F_n$ of \cite{HL} is exactly the estimator
defined in (\ref{eq:naive-estim}) (with possibly a slightly different kernel). Hence it is not always non-decreasing and cannot be used
directly to estimate $\mu$ with respect to  Wasserstein metrics.

For instance, for
a Laplace error distribution, the estimator $\hat F_n$ of \cite{HL} is such that
$$
\Big({\mathbb E}\Big [ \int |\hat F_n (t) -F(t)|^2 dt \Big] \Big)^{1/2}\leq C n^{-1/8}
 \, ,
$$
while the  rate of  convergence of our estimator for $W_1$ is
$$
 {\mathbb E}W_1(\tilde \mu_n, \mu) = {\mathbb E}\Big [ \int |\tilde F_n (t) -F(t)| dt \Big] \leq C n^{-1/5} \, .
$$
In both cases, there are no assumptions on $\mu$, except moment assumptions; in particular, $\mu$ needs 
not be absolutely continuous with respect to the Lebesgue measure. It is then a different context than 
that considered by  \cite{DGJ11} for the pointwise estimation of $F(x_0)$. In this paper, the authors always assume that $\mu$ is absolutely continuous with
respect to the Lebesgue measure, with a density $f$ belonging to a Sobolev space of
order $\alpha>-1/2$.



Note that the two rates described in this remark
are minimax (see Section \ref{sec:lowerbound} for our estimator).
\end{rem}

\begin{rem}
Since the function $H_Y(x)= P(|Y|\geq x)$ is non-increasing,
the tail condition
\begin{equation}\label{tailcond}
\int_0^\infty |x|^{p-1}\sqrt{P(|Y|\geq x)}dx<\infty
\end{equation}
in Assumption (\ref{eq:hyp})
implies that $H_Y(x)={\mathcal O}(1/ |x|^{2p})$. Hence $|Y|$ has a weak moment of order $2p$, 
which implies a strong moment of ordrer $q$ for any $q<2p$. 
Note that (\ref{tailcond}) is the same as the tail condition
(\ref{eq:tailcond}) obtained in Section \ref{sec:withouterror}
to get the rate $\E W_p^p(\mu_n,\mu)\leq Cn^{-1/2}$  in the case without noise.
Recall that, in the case without
noise when $p=1$, this condition is necessary and sufficient for the weak
convergence of $\sqrt n W_1(\mu_n, \mu)$.
Note also that 
$$
\text{(\ref{tailcond}) holds iff (\ref{eq:tailcond}) holds and}  \ \int_0^\infty |x|^{p-1}\sqrt{P(|\varepsilon|\geq x)}dx<\infty \, .
$$
The  ``if'' part follows easily from the simple inequality 
$ P(|X+\varepsilon|>x) \leq  P(|X|>x/2)+ P(|\varepsilon|>x/2)$. To prove the  ``only if''  part, note
that, since $X$ and $\varepsilon$ are independent, (\ref{tailcond}) can be written
\begin{equation}\label{fub}
\int\int_0^\infty |x|^{p-1}\sqrt{P(|X+y|\geq x)}\ dx \ \mu_{\varepsilon}(dy)<\infty \, .
\end{equation}
But this implies that 
\begin{equation}\label{X+y}
\int_0^\infty |x|^{p-1}\sqrt{P(|X+y|\geq x)}dx < \infty
\end{equation}  for $\mu_{\varepsilon}$ almost every $y$. Now
if (\ref{X+y}) holds for one $y$, then it holds for every $y$, proving that (\ref{eq:tailcond}) holds 
(and the same is true for  $\varepsilon$
by interchanging $X$ and $\varepsilon$ in (\ref{fub})). As we have seen, the tail condition on $\varepsilon$ implies that $|\varepsilon|$ has a moment of ordre $k$ for any integer $k$ strictly less than
$2p$, hence $\mu_\varepsilon^*$ is at least $k$ times continuously differentiable. 
\end{rem}

\begin{rem}
The rate $\E W_p^p(\tilde{\mu}_n, \mu)\leq C (\log n)^{-p/\beta}$
in the supersmooth case
has already been
 given in Theorem 4 of Dedecker and Michel (2013) and is valid in any dimension.
However the condition on the regularity of $r_\e$ is more restrictive
in the paper by Dedecker and Michel (2013), since it is assumed there that Condition (\ref{eq:super})
is true for $\ell\in \{0,1,\dots, \lceil p \rceil +1\}$. Note that this rate is minimax,
as stated in Theorem 2 of Dedecker and Michel (2013).
\end{rem}

\begin{rem}
Applying Proposition 1 in Dedecker and Michel (2013), if Condition (\ref{eq:ordinary})
is true
for $\ell\in \{0,1,\dots, \lceil p \rceil +1\}$,  one can build an explicit estimator
$\bar \mu_n$ such that
$\E W_p^p(\bar{\mu}_n, \mu)\leq C n^{-p/(2p+2 \beta +1)}$, which is worse than
(\ref{eq:rateordi}). The estimator $\bar \mu_n$ is the ``naive'' estimator defined in Section~\ref{sec:firststep}. However, the procedure given in Dedecker and Michel (2013) works also
when the observations $Y_i$ are ${\mathbb R}^d$-valued, whereas the estimator $\tilde \mu_n$
defined in (\ref{eq:defest}) is well defined for $d=1$ only.
 Hence, a reasonable question is: can we improve on  Proposition
 1
 of Dedecker and Michel (2013)
in any dimension?
\end{rem}

\begin{proof}
We first prove Item 1.
From Proposition $\ref{prop:gen}$ and Assumptions (\ref{eq:hyp}) and (\ref{eq:super}), we obtain the upper bound $$\E W_p^p(\tilde{\mu}_n, \mu)\leq C\left( h^p+\frac 1 {\sqrt n} \frac 1 {h^{\tilde{\beta}}}e^{1/{h^\beta\gamma}}\right).$$ Taking $h= (4/(\gamma\log(n)))^{1/\beta}$ gives the result.

We now prove Item 2.
From Proposition $\ref{prop:gen}$ and Assumptions (\ref{eq:hyp}) and (\ref{eq:ordinary}), we obtain
$$\E W_p^p(\tilde{\mu}_n, \mu)\leq \begin{cases}
 C \left(h^p+\frac 1 {\sqrt n}\frac 1{h^{\beta-1/2}}\right) & \mbox{ if }  \beta >\frac 1 2\\
C \left(h^p+\frac 1{\sqrt n}\sqrt{\log (\frac 1 h)}\right)& \mbox{ if }  \beta=\frac 1 2\\
 C\left( h^p+ \frac 1{\sqrt n}\right)& \mbox{ if }  \beta <\frac 1 2.
\end{cases}$$
Taking $h= n^{-\frac{1}{2p+(2\beta-1)_+}}$ gives the result.

\end{proof}


\section{Lower bound}\label{sec:lowerbound}

For some  $M > 0$ and $q \geq 1$, we denote by $\mathcal D(M,q)$ the set of measures $\mu$ on $\R$ such that $\int |x|^q d\mu(x)   \leq M$.
\begin{theo} \label{Theo:lowerWp}
Let $M>0$ and $q \geq 1$.
Assume that there exist $\beta>0$ and $c>0$, such that for every $\ell\in \{0,1,2\}$ and  every $t\in \R$,
\begin{equation}
\label{eq:assumplower}
|{\mu_\varepsilon^*}^{(\ell)}(t)|\leq c(1+|t|)^{-\beta}.
\end{equation}
Then, there exists a constant $C>0$ such that, for any estimator $\hat{\mu}$,
$$
\liminf_{n\to \infty} n^{\frac p{2\beta+1}}\sup_{\mu \in  \mathcal D(M,q)} \E W_p^p(\hat{\mu},\mu)>C.
$$
\end{theo}

\begin{rem} For $W_1$, this lower bound matches  the upper bound given in Theorem~\ref{theo:upperbounds} for $\beta \geq 1/2$. For $W_p$ ($p>1$), we conjecture that the upper bounds given by Theorem~\ref{theo:upperbounds}  are appropriate under the assumed tail conditions. Getting better rates of convergence for $W_p$ ($p>1$) is an open
question. From Section \ref{sec:withouterror}, it seems reasonable to think that better rates
can  be obtained when $\mu$ has an absolutely continuous component with respect to the Lebesgue
measure which is strictly positive on the support of $\mu$ (and also that this should be a necessary condition condition to reach the lower bound when $\beta>1/2$).
\end{rem}
We also give a lower bound for the cdf deconvolution problem with loss $L_p$
defined in (\ref{loss}).
\begin{theo} \label{Theo:lowercdfLp}
Let $M>0$ and $q \geq 1$.
Assume that there exist $\beta>0$ and $c>0$, such that \eqref{eq:assumplower}  is satisfied for every $\ell\in \{0,1,2\}$ and  every $t\in \R$. Then, there exists a constant $C>0$ such that, for any estimator $\hat{F}$ of $F$:
$$
\liminf_{n\to \infty} n^{\frac 1{2\beta+1}}\sup_{\mu \in  \mathcal D(M,q)} \E  \int_\R     |x|^{p-1}   |\hat F (x) - F(x) |dx   >C.
$$
\end{theo}
We give below the proof of Theorem~\ref{Theo:lowerWp} for the Wasserstein metric. The proof of Theorem~\ref{Theo:lowercdfLp}  is similar, it can be easily adapted from the proofs of Theorem~\ref{Theo:lowerWp} and of Theorem 3 in  \cite{DM13}.
\begin{proof}
Let  $M > 0$ and $q \geq 1$. The proof is similar to the proof of Theorem 3 in
\cite{DM13} and thus we only give here a sketch of the proof. We first define a finite family in $\mathcal D(M,q)$ using the densities
\begin{equation} \label{eq:f0}
 f_{0,r}(t) := C_r (1+ t^2)^{-r}
\end{equation}
with some $r >  (1+q)/2 $. Next, let $ b_n$ be the sequence
\begin{equation} \label{eq:defbn}
b_n := \Big[n^{\frac{1}{2 \beta + 1}}  \Big]\vee 1 \, ,
\end{equation}
where $[\cdot]$ is the integer part. For any $\theta \in
\{0,1\} ^{b_n}$, let
\begin{equation} \label{eq:ftheta}
f_{\theta} (t) = f_{0,r}(t) + C \sum _{s=1} ^{b_n} \theta_s  H \left(b_n ( t - t_{s,n}) \right), \quad t \in \R,
\end{equation}
where $C$ is a positive constant and $t_{s,n} = (s-1) / b_n$. The function $H$ is a bounded function whose integral on the line is $0$.
Moreover, we may choose a function $H$ such that (see for instance \cite{Fan91} or \cite{Fan93}):
\begin{description}
\item (A1) $\int_{-\infty}^{+\infty} H (t)\,  dt  = 0$ and $ \int_ 0 ^1  |H^{(-1)}(t)| \, dt >0 $,
\item (A2) $|H(t)| \leq c (1 + t^2) ^{-r_0}$ where $r_0 > \max(3/2,(1+q)/2)$,
\item (A3)  ${H^*}(z) = 0 $ outside $[1,2]$,
\end{description}
where $H^{(-1)}(t) : = \int_{-\infty} ^t H(u)  \, d u $ is a primitive of $H$. Note that by replacing $H$ by $H/C$ in the following, we finally can
take $C = 1$ in (\ref{eq:ftheta}).  Let $\mu_{\theta}$ be the measure of density $f_{\theta}$ with respect to  the Lebesgue measure.  Then we can find  some $M$
large enough such that for all $\theta \in \{0,1\} ^{b_n}$, $\mu_{\theta} \in \mathcal D(M,q)$. Moreover, under these assumptions the  first two
derivatives of $H^*$ are continuous and bounded.

For $\theta \in
\{0,1\}^{b_n}$ and $s \in \{1,\dots, b_n\}$, let us define the probability measures $\mu_{\theta,s,0}$ and $\mu_{\theta,s,1}$ with  densities
$$f_{\theta,s,0} := f_{(\theta_1, \dots,\theta_{s-1},0 ,\theta_{s+1}, \dots,\theta_{b_n})} \quad \textrm{and} \quad f_{\theta,s,1} := f_{(\theta_1,
\dots,\theta_{s-1},1,\theta_{s+1}, \dots,\theta_{b_n})}. $$
We also consider the densities $h_{\theta,s,u} = f _{\theta,s,u} \star \mu_{\varepsilon}$ for $u=0 $ or 1. Since $W_1$ is dominated by $W_p$, and using Jensen's inequality, it follows  that
\begin{multline} \label{WPW1}
 \sup _{\mu \in \mathcal D(M,q)} \E _{  (\mu \star \mu_{\varepsilon}) ^{\otimes n }}W_p ^p \left(\mu , \tilde \mu _n  \right)
\geq  \sup _{\mu \in \mathcal D(M,q)} \E _{  (\mu \star \mu_{\varepsilon}) ^{\otimes n }}W_1 ^p \left(\mu , \tilde \mu _n  \right)  \\
\geq  \left(\sup _{\mu \in \mathcal D(M,q)} \E _{  (\mu \star \mu_{\varepsilon}) ^{\otimes n }}W_1  \left(\mu , \tilde \mu _n  \right)  \right)^p \, .
\end{multline}
Using a standard randomization argument (see for the instance the proof of Theorem 3 in \cite{DM13} for the multivariate case), it can be shown that there exists a constant $C >0$ such that
\begin{equation} \label{resumBinf}
\sup _{\mu \in \mathcal D(M,q)} \E _{  (\mu \star \mu_{\varepsilon}) ^{\otimes n }}W_1 \left(\mu , \tilde \mu _n  \right)
  \geq    \frac C {b_n }  \int_{0} ^{1} \left|  H ^{(-1)}  (u)  \right| \, d  u \,
\end{equation}
as soon as there exists a constant $c >0$ such that, for any $\theta \in \{0,1\}^{b_n}$,
\begin{equation}  \label{eq:chi2h}
\chi ^2 \left(  h_{\theta,s,0}  \, , \, h_{ \theta,s,1} \right) \leq \frac c n \, ,
\end{equation}
where the $\chi ^2$
distance
between two densities $h_1$ and $h_2$ on $\R$ is defined by
$$ \chi^2(h_1,h_2)  = \int \frac{\left\{(h_1(x) -  h_2 (x)\right\}^2 }{h_1(x)} d x.$$
If (\ref{eq:chi2h}) is satisfied, we take $b_n$ as  in (\ref{eq:defbn}) and the theorem is thus proved   according to (\ref{WPW1}), (\ref{resumBinf}) and (A1).

It remains to prove (\ref{eq:chi2h}). Using (A2), we can find a constant $C>0$ such that for any $t \in \R$ and any $s \in
\{1,\dots,b_n\}$,
 \begin{align}
\chi ^2 \left(  h_{\theta,s,0} \, , \, h_{ \theta,s,1} \right)
&\leq C  b_n^{-1}   \int  \frac{ \left\{ \int  H
(v-y) \,   \mu_{\varepsilon} (d y/b_n)  \right\}^2} {  f_{0,r} \star \mu_\varepsilon (v/b_n)} d v . \label{reducekhi2}
\end{align}
The right side of (\ref{reducekhi2}) is typically the kind of  $\chi ^2$ divergence that is upper bounded in the proofs of Theorems~4 and 5 in
\cite{Fan91} for computing pointwise rates of convergence: under Assumption (\ref{eq:assumplower}), it gives that there exists a constant $C$ such
that
$$ \int  \frac{ \left\{ \int H
(v-y)  \,  \mu_{\varepsilon} (d y/b_n)   \right\}^2} {  f_{0,r} \star \mu_\varepsilon (v/b_n)} d v \leq C b_n^{-2 \beta} $$
and (\ref{eq:chi2h}) is proved.
\end{proof}

\section{Proof of Proposition \ref{prop:gen}}
\label{sec:proof}

Throughout, $C$ will denote a positive constant depending on $p$ which may change from line to line.

We start from the basic  inequality (\ref{eq:start}).
Inequality (\ref{eq:bias}) yields the bias term $$h^p2^{p-1}\int|x|^{p}k(x)dx \, , $$
and it remains to control the term $\E W^p_p(\tilde{\mu}_n,\mu\star K_h)$.

By (\ref{eq:deb}),  we have
\begin{align}\E W^p_p(\tilde{\mu}_n,\mu\star K_h)&\leq C\int|x|^{p-1}\E|\hat{F}_n-\E[\hat{F}_n]|(x)dx+\rho\nonumber\\
&\leq C\int|x|^{p-1}\sqrt{\var(\hat{F}_n)(x)}dx+\rho\label{eq:deb1}
.
\end{align}

Now, let $\phi$ denote a  symmetric function, $\lceil p \rceil$+1 times continuously differentiable, equal to 1 on the interval $[-1,1]$ and to  0 outside   $[-2,2]$. Our preliminary estimator $\hat{F_n}$ may be written
\begin{align*}\hat{F_n}(t)&=\frac 1 {nh}\int_{-\infty}^t\sum_{k=1}^n\tilde{k}_h\left(\frac{u-Y_k}{h}\right)du\\
&=\frac 1 n \sum_{k=1}^n G_{1,h}\left( \frac{t-Y_k}{h} \right)+\frac 1 n \sum_{k=1}^n G_{2,h}\left (\frac{t-Y_k}{h} \right)\\&
:=\hat{F}_{1,n}+\hat{F}_{2,n},
\end{align*}
where $$G_{1,h}(x)=\int_{-\infty}^x\tilde{k}_{1,h}(u)du \quad \mbox{ and }
\quad G_{2,h}(x)=\int_{-\infty}^x\tilde{k}_{2,h}(u)du.$$
Here,
$$
\tilde{k}_{1,h}(u)=\frac 1 {2\pi}\int \frac{e^{itu}k^*(t)\phi(t/h)}{\mu_\e^*(-t/h)}dt, \quad  \tilde{k}_{2,h}(u)= \frac 1 {2\pi}\int \frac{e^{itu}k^*(t)(1-\phi(t/h))}{\mu_\e^*(-t/h)}dt.
$$

%

From (\ref{eq:deb1}), we infer that
\begin{equation}\label{eq:decomp}
\E W_p^p(\tilde{\mu}_n,\mu\star K_h)\leq C (I+ J)+\rho,
\end{equation}
where
$$I=\int|x|^{p-1}\sqrt{\var(\hat{F}_{1,n})(x)}dx \quad
\mbox{and}
\quad
J=\int|x|^{p-1}\sqrt{\var(\hat{F}_{2,n})(x)}dx.$$

To prove Proposition \ref{prop:gen}, we shall give some upper bounds for the terms $I$ and $J$.
\paragraph{Control of  $I$.} We first split the integral into two parts:
\begin{align*}I&=
\int_{- \infty}^0|x|^{p-1}\sqrt{\var(\hat{F}_{1,n})(x)}dx+
\int_{0}^\infty|x|^{p-1}\sqrt{\var(\hat{F}_{1,n})(x)}dx:=I^-+I^+.
\end{align*}
Now,
\begin{align*}
I^-&=\int_{- \infty}^0|x|^{p-1}\sqrt{\var(\hat{F}_{1,n})(x)}dx\\
&\leq
\frac C{\sqrt{n}}\int_{-\infty}^0|x|^{p-1}\sqrt{\E\left[ G_{1,h}\left(\frac{x-Y}{h}\right)\right]^2}dx\\&\leq
 \frac C{\sqrt{n}}\int_{-\infty}^0|x|^{p-1}\sqrt{\E\left[ \int\tilde{k}_{1,h}(u)\I_{\left\{u\leq\frac{x-Y}{h} \right\}}du\right]^2}dx.
\end{align*}
Then, letting $z=uh$ and
applying Cauchy-Schwarz's inequality
we obtain, for any $a\in ]0, 1[$,
\begin{align*}
I^-&\leq
\frac C{\sqrt{n}}\int_{-\infty}^0|x|^{p-1}\sqrt{\E\left[ \int\ \frac{\tilde{k}_{1,h}(z/h)}{h}\I_{\left\{Y+z\leq x \right\}}dz\right]^2}dx\\
&\leq  \frac {C}{\sqrt{n}}
\int_{-\infty}^0|x|^{p-1}\sqrt{ \E\left[\int (1+|z|^{1+a})\left(\frac{\tilde{k}_{1,h}(z/h)}{h}\right)^2  \I_{\left\{Y+z\leq x \right\}} dz\right]}\,dx \, .
\end{align*}
Noticing that $\I_{\{Y+z\leq x \}}\leq \I_{\{Y\leq \frac x2\}}+\I_{\{z\leq \frac x2\}}$, we obtain that $I^-\leq I^-_1+I^-_2$, where
\begin{align*}
I^-_1
&=  \frac C{\sqrt{n}}\int_{-\infty}^0|x|^{p-1}\sqrt{ \E\left[\int (1+|z|^{1+a})\left(\frac{\tilde{k}_{1,h}(z/h)}{h}\right)^2  \I_{\left\{Y\leq \frac x2 \right\}} dz\right]}
dx\\
I^-_2 &= \frac C{\sqrt{n}}\int_{-\infty}^0|x|^{p-1}
\sqrt{ \int (1+|z|^{1+a})\left(\frac{\tilde{k}_{1,h}(z/h)}{h}\right)^2   \I_{\left\{z\leq \frac x2 \right\}} dz} \
dx \, .
\end{align*}

To control the term $I^-_1$, note that
\begin{align*}
I^-_1&
\leq \frac {C}{\sqrt{n}}\sqrt{ \int (1+|z|^{2})\left(\frac{\tilde{k}_{1,h}(z/h)}{h}\right)^2    dz}\int_{-\infty}^0|x|^{p-1}\sqrt{{\mathbb P}\left(Y\leq \frac x 2\right)}dx .
\end{align*}
Here we shall use the following lemma.

\begin{lem}\label{lem:trivial}
For any nonnegative integer $k$ and any $h\leq 1$ we have
$$
\int |z|^{2k}\left(\frac{\tilde{k}_{1,h}(z/h)}{h}\right)^2  dz
\leq C\Big(\sup_{t \in [-2,2]}\sum_{\ell=0}^{k}|r_\e^{(\ell)}(t)| \Big)^2 \, .
$$
\end{lem}

\begin{proof} By definition of $\tilde{k}_{1,h}$,
$$
\int |z|^{2k}\left(\frac{\tilde{k}_{1,h}(z/h)}{h}\right)^2  dz\leq \frac 1 {4\pi^2}
\int |z|^{2k}\left|\int\frac{e^{iu z}k^*(uh)\phi(u)}{\mu_\e^*(-u)}du\right|^2
dz \, .
$$
Now, by Parseval-Plancherel's identity,
$$\int|z|^{2k}\left|\int \frac{e^{iuz}k^*(uh)\phi(u)}{\mu_\e^*(-u)}du\right|^2dz=2\pi
\int \left|
\left(\frac{k^*(th)\phi(t)}{{\mu_\e^*}(-t)}\right)^{(k)}\right|^2dt.$$
It can be checked that, for $h\leq 1$,
$$
\left|
\left(\frac{k^*(th)\phi(t)}{{\mu_\e^*}(-t)}\right)^{(k)}\right| \leq C\sum_{\ell=0}^{k}|r_\e^{(\ell)}(t)|\I_{[-2,2]}(t),
$$
which concludes the proof of the Lemma.
\end{proof}
Applying Lemma \ref{lem:trivial} with $k=1$, we obtain that
\begin{align}\label{eq:I1}
I^-_1&
\leq \frac {C}{\sqrt{n}}\Big( \sup_{t \in [-2,2]}\sum_{\ell=0}^{1}|r_\e^{(\ell)}(t)|    \Big)\int_{-\infty}^0|x|^{p-1}\sqrt{{\mathbb P}\left(Y\leq \frac x 2\right)}dx .
\end{align}

We now control the term $I_2^-$. Let $b\in]0,1[$. Applying Cauchy-Schwarz's inequality
\begin{align*}
I^-_2&\leq
\frac {C}{\sqrt{n}}
\sqrt{\int_{-\infty}^0|x|^{2p-2}(1+|x|^{1+b}) \int_{-\infty}^\frac x2(1+|z|^{1+a})\left(\frac{\tilde{k}_{1,h}(z/h)}{h}\right)^2dz\,dx}.
\end{align*}
Consequently, by Fubini's Theorem
\begin{align*}
I_2^-&\leq \frac {C}{\sqrt{n}}
\sqrt{\int_{-\infty}^0 (1+|z|^{1+a})\left(\frac{\tilde{k}_{1,h}(z/h)}{h}\right)^2
\int_{2z}^0|x|^{2p-2}(1+|x|^{1+b})dx\,dz}
 \\& \leq \frac {C}{\sqrt{n}}
\sqrt{\int (1+|z|^{2p+1+a+b})\left(\frac{\tilde{k}_{1,h}(z/h)}{h}\right)^2
dz}
\end{align*}
Let $m_0$ be the least integer strictly greater than $p+1/2$.
Taking $a$ and $b$ close enough to $0$, it follows that
$$
I_2^- \leq \frac {C}{\sqrt{n}}
\sqrt{\int (1+|z|^{2m_0})\left(\frac{\tilde{k}_{1,h}(z/h)}{h}\right)^2 dz }
$$
Applying Lemma \ref{lem:trivial} with $k=m_0$, it follows that
\begin{equation}\label{eq:I2}
I_2^- \leq \frac {C}{\sqrt{n}} \Big(\sup_{t \in [-2,2]}\sum_{\ell=0}^{m_0}|r_\e^{(\ell)}(t)| \Big)\, .
\end{equation}
In the same way, we have
\begin{align*}
I^+&=\int_{0}^\infty|x|^{p-1}\sqrt{\var(1-\hat{F}_{1,n})(x)}dx\\
&\leq
\frac C{\sqrt{n}}\int_0^\infty|x|^{p-1}\sqrt{\E\left[1- G_{1,h}\left(\frac{x-Y}{h}\right)\right]^2}dx\\&\leq
 \frac C{\sqrt{n}}\int_0^\infty|x|^{p-1}\sqrt{\E\left[ \int\tilde{k}_{1,h}(u)\I_{\left\{u\geq\frac{x-Y}{h} \right\}}du \right]^2}dx.
\end{align*}
Using the same arguments as for $I^-$, we  obtain,
\begin{multline}\label{eq:Iplus}
I^+\leq \frac {C}{\sqrt{n}}\Big(\sup_{t \in [-2,2]}\sum_{\ell=0}^{1}|r_\e^{(\ell)}(t)| \Big)\int_0^\infty |x|^{p-1}\sqrt{P\left(Y\geq \frac x 2\right)}dx  + \frac {C}{\sqrt{n}} \Big(\sup_{t \in [-2,2]}\sum_{\ell=0}^{m_0}|r_\e^{(\ell)}(t)| \Big) \, .
\end{multline}
Consequently, gathering (\ref{eq:I1}), (\ref{eq:I2}) and (\ref{eq:Iplus}) we obtain that
\begin{multline}\label{eq:I}
I\leq \frac {C}{\sqrt{n}}\Big(\sup_{t \in [-2,2]}\sum_{\ell=0}^{1}|r_\e^{(\ell)}(t)| \Big)\int_0^\infty |x|^{p-1}\sqrt{P\left(|Y|\geq  x \right)}dx + \frac {C}{\sqrt{n}} \Big(\sup_{t \in [-2,2]}\sum_{\ell=0}^{m_0}|r_\e^{(\ell)}(t)| \Big) \, .
\end{multline}

\paragraph{ Control of  $J$.}
Let $a \in ]0,1/2[$. By definition of the term $J$, and applying Cauchy-Schwarz's inequality,
\begin{align*}J
&\leq \frac C{\sqrt{n}}\int|x|^{p-1}\sqrt{\E\left[ G_{2,h}\left(\frac{x-Y}{h}\right)\right]^2}dx\\
&\leq \frac C{\sqrt{n}}\sqrt{\int|x|^{2p-2}(1+|x|^{1+a})\E \left[ G_{2,h}\left(\frac{x-Y}{h}\right)\right]^2dx}.
\end{align*}
Let us write $$G_{2,h}(x)=\lim_{T\to -\infty}\int_T^x\tilde{k}_{2,h}(u)du=
\lim_{T\to -\infty}\int \I_{[T,x]}(u)\tilde{k}_{2,h}(u)du.$$
Using  Parseval-Plancherel's identity, we get
\begin{align*}G_{2,h}(x)&=\lim_{T\to -\infty}\frac 1{2\pi}\int \overline{\I^*_{[T,x]}}(u)\tilde{k}^*_{2,h}(u)du
\\&=-\frac{1}{2\pi i}\left[\int \frac {e^{-itx}}{t}\frac {k^*(t)(1-\phi(t/h))}{\mu_\e^*(t/h)}dt-\lim_{T\to -\infty}\int \frac {e^{-itT}}{t}\frac {k^*(t)(1-\phi(t/h))}{\mu_\e^*(t/h)}dt\right]\end{align*}
Since the function $t\mapsto \dfrac {k^*(t)(1-\phi(t/h))}{t\mu_\e^*(t/h)} $ is integrable, the Riemann-Lebesgue Lemma ensures that $$ \lim_{T\to -\infty}\int \frac {e^{-itT}k^*(t)(1-\phi(t/h))}{t\mu_\e^*(t/h)}dt=0,$$ so that  $$G_{2,h}(x)=-\frac{1}{2\pi i}\int  \frac{e^{-itx}k^*(t)(1-\phi(t/h))}{t\mu_\e^*(t/h)}dt.
$$
Consequently,
\begin{align*}J\leq\frac C{\sqrt{n}}\sqrt{\E\left(\int(1+|x|^{2p-1+a})\left( -\frac 1{2\pi i}\int \frac{e^{-it\frac{x-Y}{h}}k^*(t)(1-\phi(t/h))}{t\mu_\e^*(t/h)}dt\right)^2dx\right)} .
\end{align*}
Setting $u=t/h$ and using the fact that $|x|^{q} \leq 2^{q-1}|x-Y|^{q} + 2^{q-1}|Y|^{q}$
for any $q\geq 1$, we obtain that
\begin{align*}
J&\leq\frac C{\sqrt{n}}\left[{\mathbb E}\left ( \int |x-Y|^{2p-1+a}\left(-\frac 1{2\pi i}\int \frac{e^{-iu(x-Y)}k^*(uh)(1-\phi(u))}{u\mu_\e^*(u)}du\right)^2dx
\right )
\right.
\\&
\left.\quad   +\E \left ((1+ |Y|^{2p-\frac 1 2 })\int  \left(-\frac 1{2\pi i}\int \frac{e^{-iu(x-Y)}k^*(uh)(1-\phi(u))}{u\mu_\e^*(u)}du\right)^2dx
\right )
\right]^{1/2}.
\end{align*}
Thus,
\begin{align*}
J&\leq\frac C{\sqrt{n}}\left[\int (1+ |x|^{2p-1 + a}) \left(-\frac 1{2\pi i}\int \frac{e^{-iux}k^*(uh)(1-\phi(u))}{u\mu_\e^*(u)}du\right)^2dx
\right.
\\&
\left.\quad \quad \quad \quad \quad +\E |Y|^{2p- \frac 1 2 }\int  \left(-\frac 1{2\pi i}\int \frac{e^{-iux}k^*(uh)(1-\phi(u))}{u\mu_\e^*(u)}du\right)^2dx
\right]^{1/2}.
\end{align*}
Let $m_1$ be the least integer strictly greater than $p- \frac 1 2$.
Taking $a$ close enough to zero,  it follows that
\begin{align*}
J&\leq\frac C{\sqrt{n}}\left[\int(1+|x|^{2 m_1})\left|
\frac 1{2\pi }\int \frac{e^{-iux}k^*(uh)(1-\phi(u))}{u\mu_\e^*(u)}du\right|^2dx
\right.
\\&
\left.\quad \quad \quad \quad \quad +\E |Y|^{2p- \frac 1 2 }\int  \left| \frac 1{2\pi }\int \frac{e^{-iux}k^*(uh)(1-\phi(u))}{u\mu_\e^*(u)}du \right|^2dx
\right]^{1/2}.
\end{align*}
 By Parseval-Plancherel's identity, $$\int\left|\int \frac{e^{-iux}k^*(uh)(1-\phi(u))}{
u\mu_\e^*(u)}du\right|^2dx=2\pi
\int
\left|\frac{k^*(th)(1-\phi(t))}{t\mu_\e^*(-t)}\right|^2dt,$$ and
\begin{multline*}
\int|x|^{{2 m_1 }}\left|\int \frac{e^{-iux}k^*(uh)(1-\phi(u))}{
u\mu_\e^*(u)}du\right|^2dx
=2\pi
\int
\left|\left(\frac{k^*(th)(1-\phi(t))}{t\mu_\e^*(-t)}\right)^{({m_1})}\right|^2dt.
\end{multline*}
Now, for $h\leq 1$, \begin{align*}\left|\left(\frac{k^*(th)(1-\phi(t))}{t\mu_\e^*(-t)}\right)^{(m_1)}\right|&\leq
C\sum_{j=0}^{m_1}\sum_{\ell=0}^j \frac{|r_\e^{(\ell)}(-t)|}{|t|^{j-\ell+1}}\I_{[-1,1]^c}(t)
\\&\leq C\sum_{\ell=0}^{m_1}
\frac{|r_\e^{(\ell)}(-t)|}{|t|}\I_{[-1,1]^c}(t).
\end{align*}

Finally,
\begin{multline} \label{J}
J\leq \frac C{\sqrt{n}} \left [\E |Y|^{2p- \frac 1 2 }\int_{-1/h}^{1/h}
\frac{|r_\e(x)|^2}{|x|^2}\I_{[-1,1]^c}(x)dx + \sum_{\ell=0}^{m_1}\int_{-1/h}^{1/h}
\frac{|r_\e^{(\ell)}(x)|^2}{|x|^2}\I_{[-1,1]^c}(x)dx \right ]^{1/2}.
\end{multline}

Starting from (\ref{eq:start}) and gathering the upper bounds
(\ref{eq:bias}), (\ref{eq:decomp}), (\ref{eq:I}) and (\ref{J}),
the proof of Proposition \ref{prop:gen} is complete.

\section{Numerical experiments}\label{sec:simu}

This section is devoted to the implementation of the deconvolution estimators. We continue the experiments of \cite{CCDM11} about Wasserstein deconvolution
in the ordinary smooth case. In particular, we study the $W_1$ and $W_2$ univariate deconvolution problems and we compare our numerical
results with the upper and lower bounds given in the previous sections. We also apply our procedure to the deconvolution of the
uniform measure on the Cantor set. The deconvolution method is implemented in R.

\subsection{Implementation of the deconvolution estimators}
\label{sec:firststep}
For all the experiments we  use the kernel
$$
 k(x)=\frac{3}{16 \pi}
  \left(\frac{8\sin(x/8)}{x}\right)^4
$$
which corresponds to the kernel given by (\ref{eq:kernelp}) with $p=2$ and a Fourier support over $[-1/2,1/2]$. Computing the deconvolution estimators
requires to evaluate many times the function
$$
 \tilde{k}_h : x \mapsto \frac 1 {2\pi}\int\frac {e^{iux }k^*(u)}{\mu_\e^*(-u/h)}du
 $$
which is  the Fourier transform of
$$ \psi_h :  u \mapsto  \frac 1 {2 \pi}  \frac {k^*(u)}{\mu_\e^*(-u/h)}.$$
The Fourier decomposition of $ \psi_h$  is given by
$ \psi_h(u) =   \sum_{k \in \mathbb Z } a_{k,h} e^{ 2i \pi k u } $
where
$
a_{k,h}  =  \int_{-1/2}^{1/2}  \psi_h(u) e^{ -2 i \pi k u } du  .
$
In this section we consider symmetric distributions for $\mu_\e$. Thus $k^*$  and $\mu_\e^*$ are  even functions, 
and the $ a_{k,h}$'s are real coefficients. 
Next,
\begin{eqnarray*}
 \tilde{k}_h (x) &=& \int_{-1/2}^{1/2} \psi_h(u) e^{  i xu } du \\
 & = & \sum_{k \in \mathbb Z } a_{k,h}  \int_{-1/2}^{1/2} e^{ i(2 \pi k  + x   ) u }  du \\
  & = &  \sum_{k \in \mathbb Z } a_{k,h} \, {\rm sinc} \left( \frac{2 \pi k +x}{2}\right).
\end{eqnarray*}
For large $N$, the coefficient $a_{k,h}$ can be approximated by the $k$-th coefficient of a discrete Fourier transform taken at
$\left(\psi_h(0),\psi_h(1/N),\dots, \psi_h(1-1/N)\right)$, denoted $\hat a _{k,h,N}$ in the sequel. Of course  we use  the Fast Fourier Transform
algorithm to compute these quantities.
For some large $K$, we evaluate $ \tilde{k}_h$ at some point $x$ by
\begin{equation}
 \label{eq:ktildehat}
  \hat{\tilde{k}}_h(x) \approx \sum_{|k| \leq K }  \hat a _{k,h,N} \, {\rm sinc} \left( \frac{2 \pi k +x}{2}\right).
\end{equation}
For intensive simulation, it may be relevant to preliminary  compute  $\hat{\tilde{k}}_h$ on a grid of high  resolution rather than calling this function
each time.

We first define a discrete approximation of the function
\begin{equation*} \label{eq:naive-estim-mu}
\hat{\mu}_{n,h}: u \mapsto \frac 1 {nh} \sum_{k=1}^n  \tilde{k}_h \left(\frac{u-Y_k}{h}\right) .
\end{equation*}
Let $\mathcal  P= \{ t_1 < \dots < t_q\} $  be a finite regular grid of points in $\R$
with resolution $\eta$. A discrete approximation $  \hat \mu^d_{n,h}  $ of  $\hat \mu_{n,h}$ is defined on $\mathcal P$ by
$$ \hat \mu^d_{n,h}=  \eta \sum_{j=1}^q \hat{\mu}_{n,h}(t_j) \delta_{t_j} ,$$
where $\delta_{x}$ is the Dirac distribution at $x$. Since $\hat{\mu}_{n,h}(t_j) $ can be negative, the first method for estimating $\mu$ consists in taking the positive part of $\hat{\mu}_{n,h}^d$ :
$$\hat \mu^{\mbox{\tiny naive}}_{n,h}   :=  \frac{ \sum_{j =1}^q    \left( \hat{\mu}_{n,h}^d (t_j)  \right)^+ \delta_{t_j }}{ \sum_{j =1}^q    \left(
\hat{\mu}_{n,h}^d (t_j) \right)^+ } . $$
This first estimator is called the ``naive'' deconvolution estimator henceforth. Note
that it was studied in \cite{CCDM11} and \cite{DM13}. For implementing the alternative estimator $\tilde \mu_{n,h}$ proposed in this paper, we first need to find some probability distribution $\tilde F_{n,h}$ on $\R$ such that
\begin{multline}
\int_\R|x|^{p-1}|\hat{F}_{n,h}-\tilde{F}_{n,h}|(x)dx  \\
\approx  \inf\left\{ \int_\R|x|^{p-1}|\hat{F}_{n,h}-G|(x)dx \, |,  \ \, G \textrm{ probability distribution on } \R \right\}.
\label{eq:defest1}
\end{multline}
In practice, this corresponds to finding a distribution function close to the step function
$$\hat F^d_{n,h} :t \mapsto    \sum_{j =1}^q      \hat{\mu}_{n,h}^d (t_j)   \I_{\{ t_j \leq t \}} . $$
Since $\hat F^d_{n,h}$ may take its values outside $[0,1]$, we can also look for a distribution function
close to $t \mapsto \hat F^d_{n,h}(t) \I_{\hat F^d_{n,h} (t) \in [0,1]}$. In other terms, we  compute the isotone regression of $t \mapsto \hat
F^d_{n,h}(t) \I_{\hat F^d_{n,h} (t) \in [0,1]}$ with weights $t_j^{p-1}$:
$$\hat F^{\mbox{\tiny isot},p}_{n,h} :=  \mbox{argmin} \left\{
  \sum_{j=1}^q  |t_j|^{p-1}
 \left|
G(t_j) -  \hat F^d_{n,h}(t_j) \I_{\hat F^d_{n,h} (t_j) \in [0,1]}
 \right|^p
   \,  ,  \textrm{ $G$ non-decreasing}
\right\} . $$
We compute $\hat F^{\mbox{\tiny isot},p}_{n,h}$ thanks to the function {\ttfamily gpava}  from the R package {\ttfamily isotonic} \citep{mair2009isotone}. The measure $\mu$ is finally estimated by the absolutely continuous  measure  $ \hat \mu^{\mbox{\tiny isot},p}_{n,h}$  whose distribution function is $\hat F^{\mbox{\tiny isot},p}_{n,h}$. We call this estimator the isotone deconvolution estimator for the metric $W_p$.

The construction of $\hat \mu^{\mbox{\tiny isot},p}_{n,h}$ depends on many parameters, for instance $K$, $h$, $N$ and $\eta$. Tuning all these parameters is a tricky issue.
For this paper we only tune these quantity by hand. The bandwidth choice is discussed in Section~\ref{sub:choiceh}. Note that one crucial point is the length $N$ of the vector we use for
computing the the $a _{k,h,N}$'s with the FFT. For ordinary smooth distributions, we observe that $\tilde k_h$ decreases slowly for small
 $\beta$ for the range of bandwidths $h$ giving minimum Wasserstein risks. Consequently, a small $\beta$ requires many terms in the expansion (\ref{eq:ktildehat}),  and hence a large $N$. For $\beta$ smaller than 0.5, it was necessary to take $N \approx 10^4$.

\subsection{Computation of Wasserstein risks for simulated experiments}
\label{sec:simrates}
\begin{figure}
\begin{center}
	\includegraphics[scale=1.2]{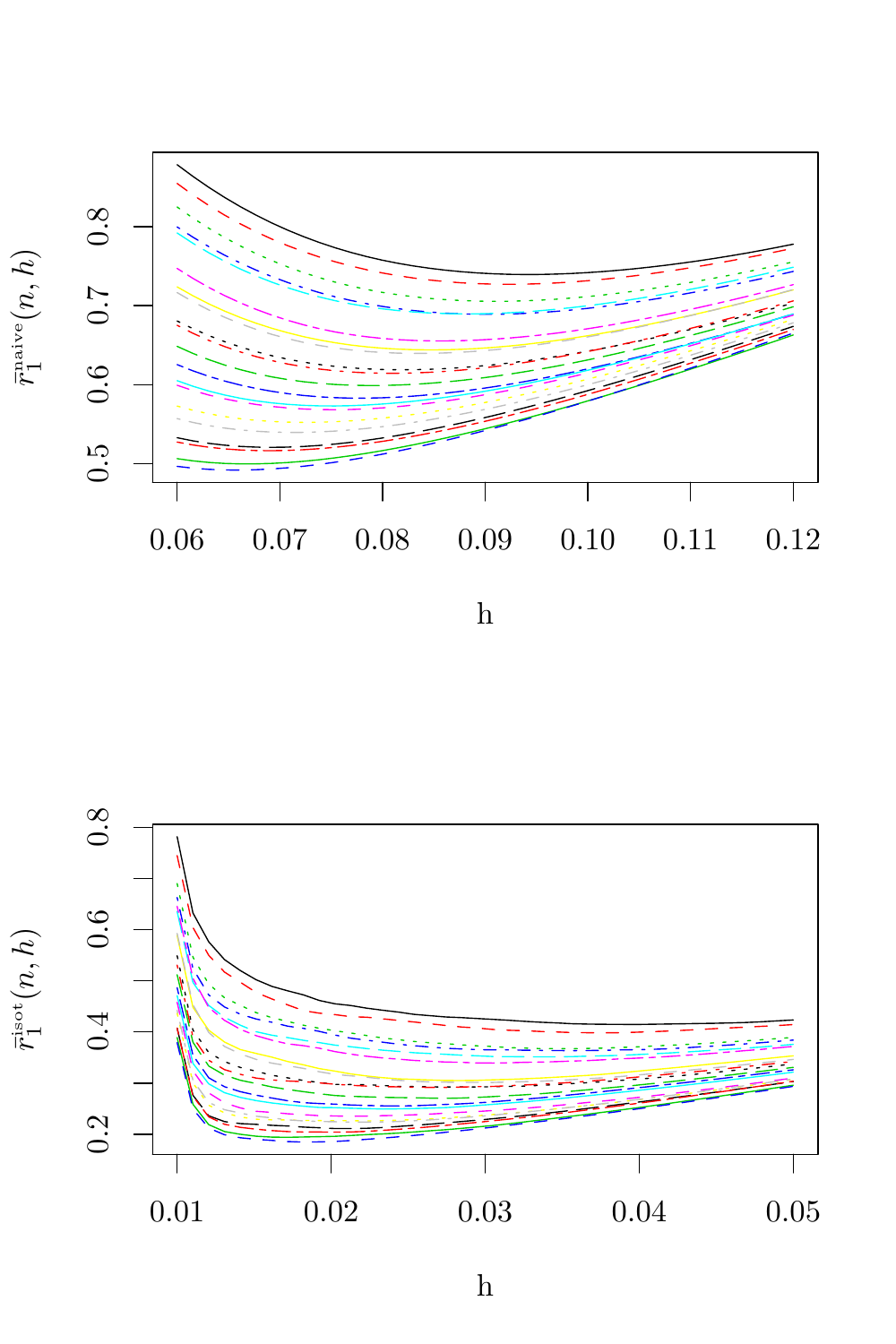}
\caption{Estimated Wasserstein risks for the Dirac experiment. The noise distribution is the symmetrized  Gamma distribution with $\beta = 2$. The twenty curves correspond to samples of length $n$ taken between 100 and 2000.}
\label{fig:Gamma2dot0}
\end{center}
\end{figure}

For fixed distributions $\mu$ and $\mu_\varepsilon$,  we simulate $Y_1,\dots,Y_n$ according to the convolution model (\ref{ModeConvo}). For a given bandwidth $h$ and $p \geq 1$, we can compute  $W_p^p ( \hat \mu^{\mbox{\tiny naive}}_n , \mu) $ and  $W_p ^p( \hat \mu^{\mbox{\tiny isot},p}_{n,h} ,\mu) $ using the quantile
 functions of the measures, thanks to the relation (\ref{WpFinv}).
The Wasserstein risks   $ \mathcal R^{\mbox{\tiny naive}} (n,h) := \E W_p^p ( \hat \mu^{\mbox{\tiny naive}}_{n,h} , \mu)$ and  $  \mathcal R^{\mbox{\tiny isot}} (n,h) := \E W_p^p ( \hat \mu^{\mbox{\tiny isot},p}_{n,h} , \mu) $ can be estimated by an elementary Monte Carlo method by repeating the simulation of the $Y_i$'s and averaging the Wasserstein distances. Let $ {\bar r }_p^ {\mbox{\tiny isot}} (n,h)$ and $ {\bar r} _p^{\mbox{\tiny naive}} (n,h)$ be the estimated risks obtained this way (see Figure \ref{fig:Gamma2dot0} for an illustration of such curves for the Dirac experiment). For each $n$, an approximation  of the minimal risks over the bandwidths is proposed by
$$ \bar r  ^{\mbox{\tiny isot}}_{p,*}(n) :=\mbox{min}_{h \in \mathcal H}  {\bar r }_p^ {\mbox{\tiny isot}} (n,h) $$
and
$$ \bar r^{\mbox{\tiny naive}}_{p,*}(n) :=\mbox{min}_{h \in \mathcal H} {\bar r} _p^{\mbox{\tiny naive}} (n,h)$$
where $\mathcal H$ is a grid of bandwidth values.

\subsection{Estimation of the rates of convergence} \label{sub:RatesEstim}

In this experiment we study the rates of convergence of the  estimators for the deconvolution of three distributions:
\begin{itemize}
\item Dirac distribution at 0,
\item Uniform distribution on $[-0.5,0.5]$,
\item Mixture of the Dirac distribution at 0 and the uniform distribution on $[-0.5,0]$.
\end{itemize}
We take  for $\mu_{\varepsilon}$ the ordinary smooth distributions summarized in Table \ref{tab:distDirac}. Recall that the coefficient $\beta$ of a symmetrized Gamma distribution is twice the shape parameter of the distribution.
\begin{table}
\begin{center}
\begin{tabular}{|c|c|c|}
\hline
Distribution   &  $\mu_{\varepsilon}^*$  & $\beta$ \\
\hline
Symmetrized Gamma & $ t \mapsto (1+   t^2  )^{-\beta/2} $&  0.3, 0.5, 1.2, 2 ,3, 4   \\
Laplace  & $ t \mapsto (1+  t^2)^{-1} $& 2 \\
Symmetrized $\chi ^2$ &  $t \mapsto (1+  4  t^2)^{(-1/2)}$ & 1\\
\hline
\end{tabular}
\end{center}
\caption{Ordinary smooth distributions used for the error.}
\label{tab:distDirac}
\end{table}
For each error distribution and for $n$ chosen between 100 and 2000, we simulate 200 times a sample of length $n$ from which we compute  the estimated minimal risks $ \bar r^{\mbox{\tiny isot}}_{p,*}(n)$ and $ \bar r^{\mbox{\tiny naive}}_{p,*}(n) $. We study the Wasserstein risks $W_1$ and $W_2$. We obtain some estimation of the exponent of the rate of convergence for each deconvolution problem by computing the linear regression of $\log \bar r_{p,*}(n)$  by $\log n$. See Figure~\ref{fig:Gamma2dot0LogLog} for an illustration and Figures~\ref{fig: LogLogW1} and \ref{fig: LogLogW2} at the end of the paper for the complete outputs of the Dirac case. A linear trend can be observed in all cases.
As expected, the risks are smaller for the isotone estimators than for the naive ones.
\begin{figure}
\begin{center}
	\includegraphics[scale=1]{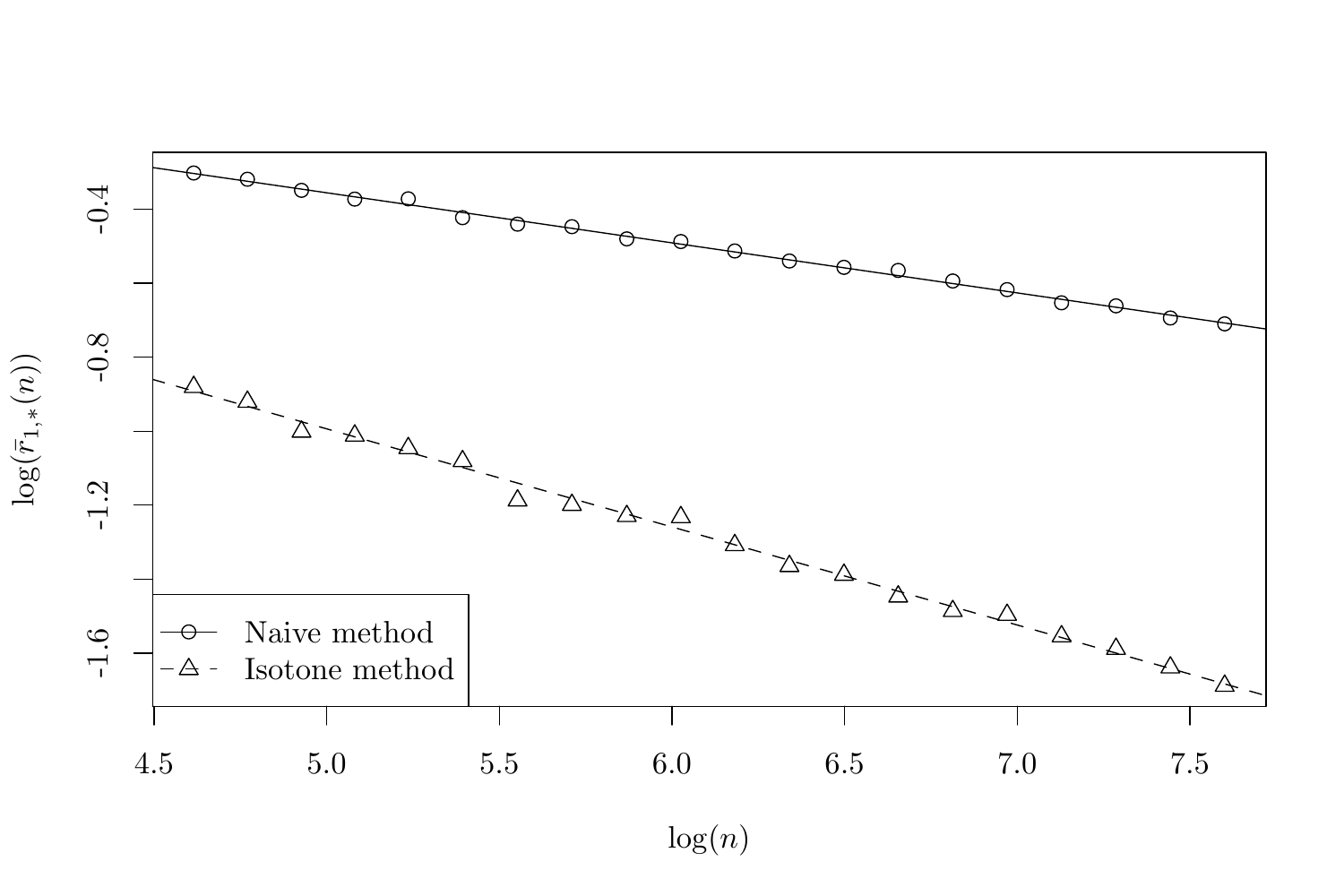}
\caption{Estimated rates of convergence to zero of the $W_1$-risk for the naive method and the isotone method for $\mu$ being a Dirac distribution at 0. The noise distribution is the symmetrized Gamma distribution  with $\beta = 2$.}
\label{fig:Gamma2dot0LogLog}
\end{center}
\end{figure}

The estimated exponents of the convergences rates are plotted in Figure \ref{fig:diracRates} as functions of $\beta$. These estimated rates can be compared with the upper and lower  bounds obtained in the paper.  Of course the rates of convergence of the isotone estimator have no reason to match exactly  the lower bounds. However it can be checked that the estimated rates we obtain are consistent with the theoretic bounds proved before. In particular we see that the parametric rate is reached for values of $\beta$ close to 0, at least in the Dirac case. These results also suggest that the correct minimax rate for $W_2$  probably corresponds to the upper bound given in Theorem \ref{theo:upperbounds} (that is,  when no further assumption is made on the
unknown distribution $\mu$).
\begin{figure}
\begin{center}
	\includegraphics[scale=0.9]{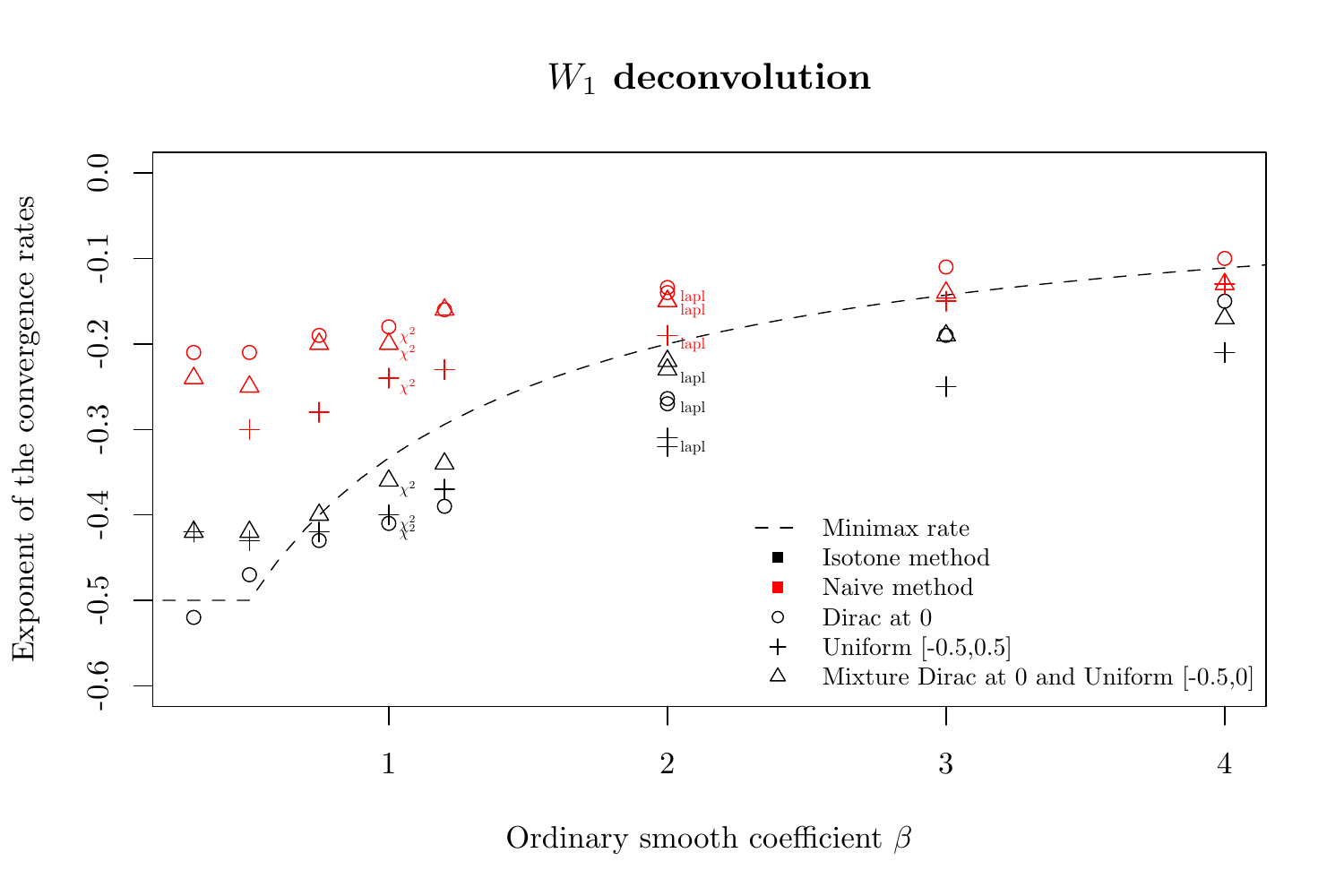}
	\includegraphics[scale=0.9]{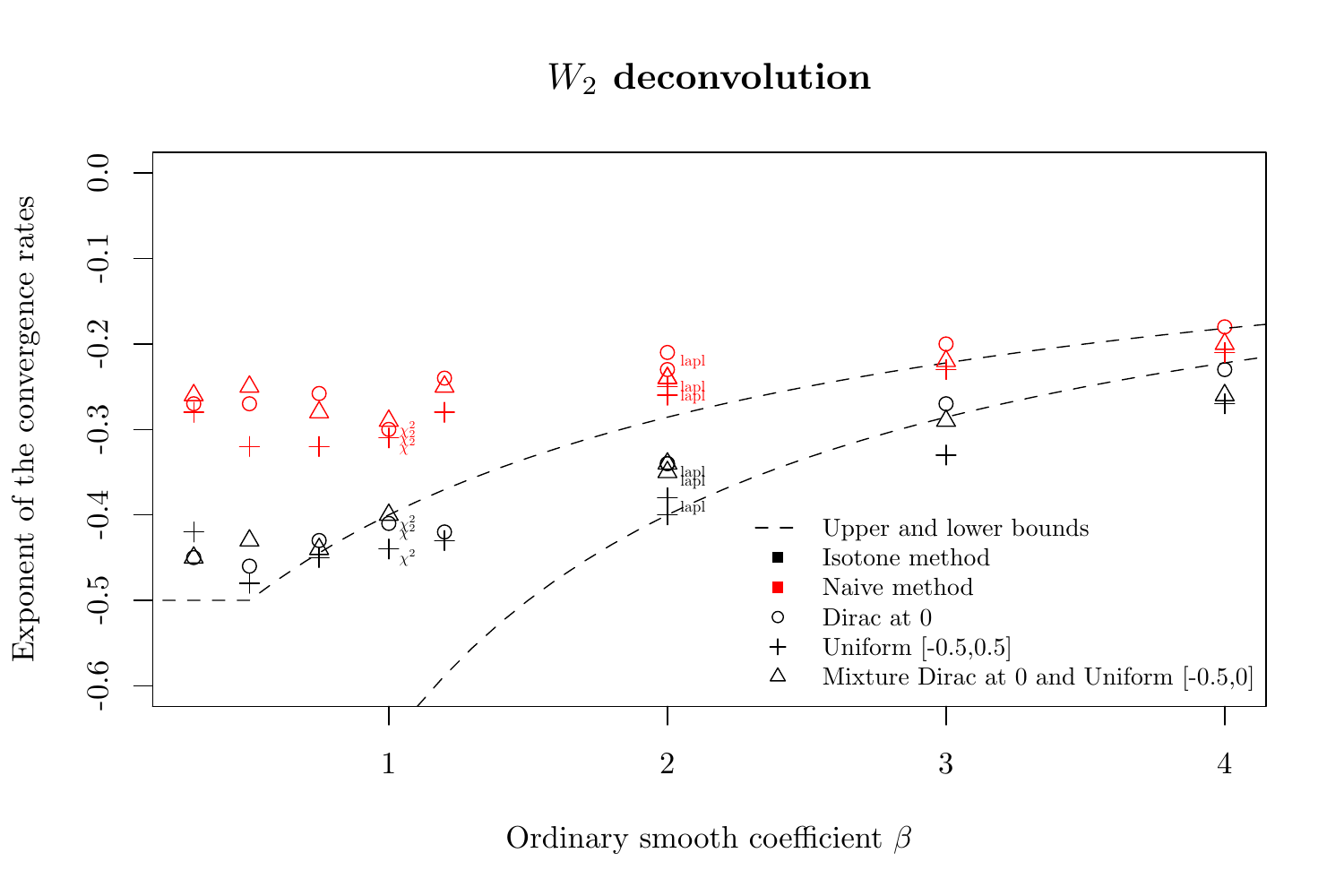}	
\caption{Estimated exponents of the convergence rates of the naive and the isotone deconvolution estimator for three distributions $\mu$. The exponents are given as functions of the ordinary smooth coefficient $\beta$. Regarding the noise distribution, the $\chi^2$ and the Laplace distributions are indicated directly on the graph, the others experiments have been done with the symmetrized Gamma distribution. The top graph corresponds to the $W_1$ deconvolution and the bottom one to the $W_2$ deconvolution.}
\label{fig:diracRates}
\end{center}
\end{figure}

\subsection{Cantor set experiment} \label{exp:Cantor}

\begin{figure}
\begin{center}
	\includegraphics[width=1 \textwidth]{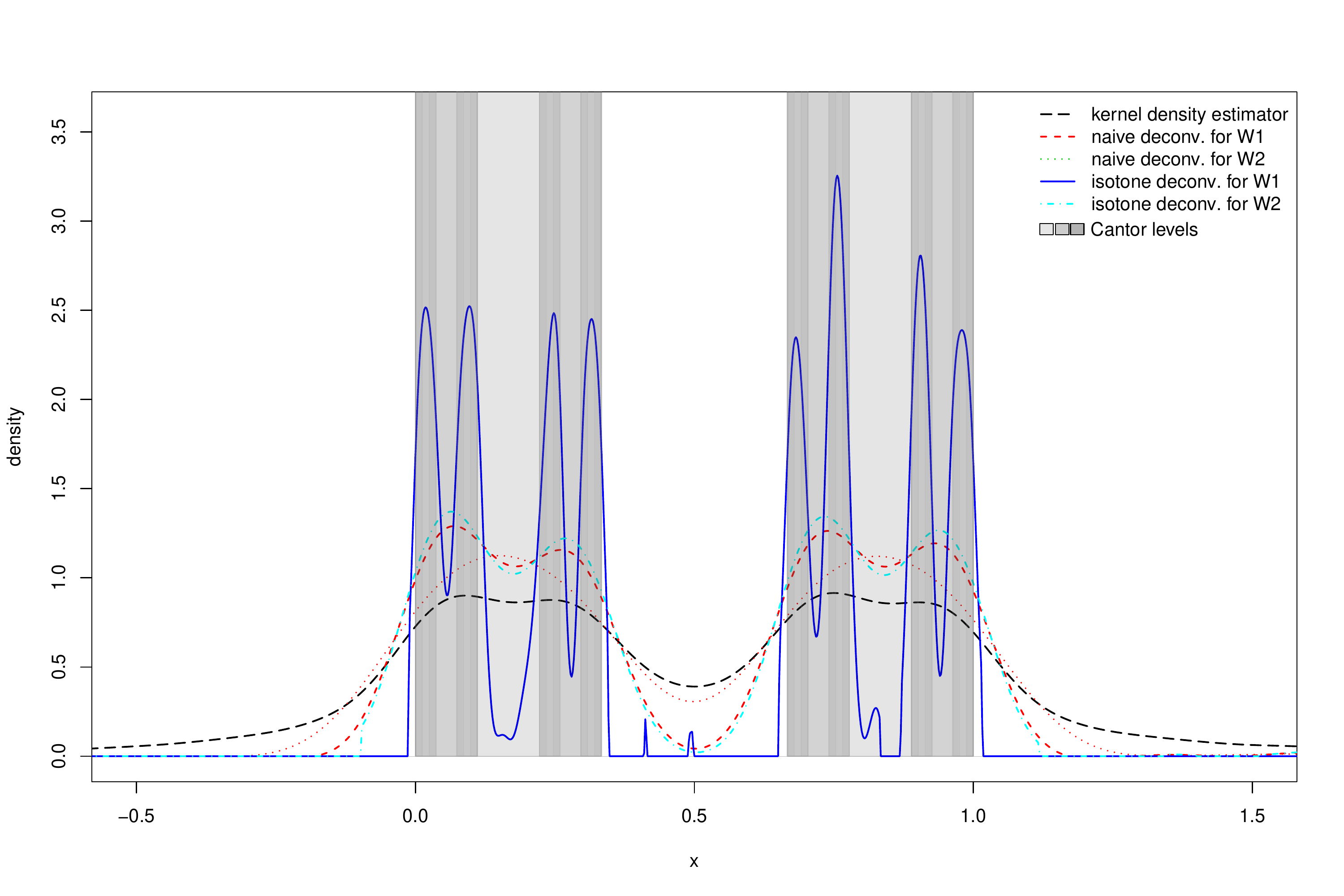}
\caption{Deconvolution of the uniform measure on the Cantor set.}
\label{fig: Cantor}
\end{center}
\end{figure}

 \begin{figure}
\begin{center}
	\includegraphics[width=0.7 \textwidth]{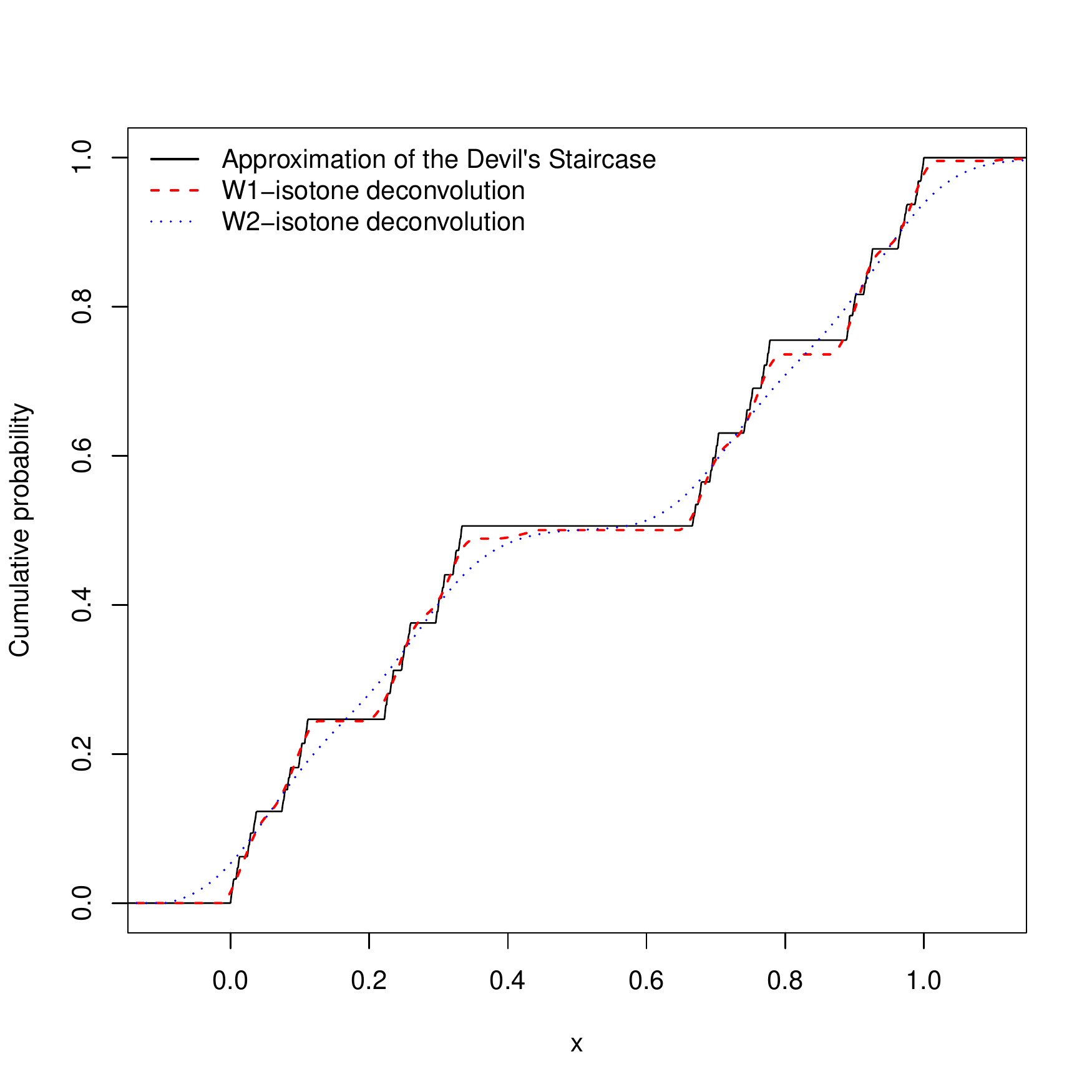}
\caption{Approximation of the Devil's staircase and distributions functions of the $W_1$ and $W_2$ isotone deconvolution estimators.}
\label{fig:Devilstaircase}
\end{center}
\end{figure}

We now illustrate the deconvolution method with a more original experiment. We take for $\mu$ the uniform distribution on the Cantor set {\gothfamily C}.
 Remember that the Cantor set can be defined by repeatedly deleting the open
  middle thirds of a set of line segments:
 $$ \mbox{\gothfamily C}=  \bigcap_{m \geq 1 } F_n$$
where $F_0 = [0, 1]$ and  $F_{m+1}$ is obtained by cutting out the middle thirds of all the intervals of $F_m$:
$F_1 = [0,\frac 1 3] \cup [ \frac 2 3,1]$ and $F_2 = [0, \frac 1 9] \cup [\frac 2 9 , \frac 1 3] \cup [\frac 2 3,\frac 7 9] \cup [\frac 8 9 , 1]$,  etc... The uniform measure $\mu_{\mbox{\gothfamily C}}$  on  \mbox{\gothfamily C} can be defined as the distribution of the random variable $X := 2 \sum_{k \geq 1} 3 ^{-k} B_k$ where  $(B_k)_{k \geq 1}$ is a sequence of independent random variables with Bernoulli distribution of parameter $1/2$.  Note that the Lebesgue measure of  {\gothfamily C}  is zero and thus the Lebesgue measure and $\mu_{\mbox{\gothfamily C}}$ are singular. The deconvolution estimators being densities for the Lebesgue measure, the Wasserstein distances are relevant metrics for comparing these  with $\mu_{\mbox{\gothfamily C}}$.

Let $\mu_{\mbox{\gothfamily C}, K}$ be the distribution of the random variable defined by the partial sum $\tilde X := 2 \sum_{k = 1} ^K  3 ^{-k} B_k$ where the $B_k$'s are defined as before. The distribution $\mu_{\mbox{\gothfamily C},  K}$ is an approximation of $\mu_{\mbox{\gothfamily C}}$ which can be computed in practice.  We simulate a sample of $n= 10^4$ observations from $\mu_{\mbox{\gothfamily C}, K}$ with $K = 100$.  These observations are contaminated by  random variables with symmetrized Gamma distribution
(the shape parameter is equal to 1/4 (so that $\beta = 0.5$) and the scale parameter
is equal to 1/2).

In Figure \ref{fig: Cantor},  the isotone estimators for $W_1$ and $W_2$ and the naive estimator are plotted on the first four levels $F_m$ of the Cantor set. The bandwidths are chosen by minimizing the Wasserstein risks over a grid, as in Section \ref{sub:RatesEstim}. This requires to approximate the quantile functions for the isotone deconvolution estimator and for the $\mu_{\mbox{\gothfamily C}}$. Regarding the quantile function of  $\mu_{\mbox{\gothfamily C}}$, we simulate a large sample according to $\mu_{\mbox{\gothfamily C},100}$  and we compute the corresponding empirical distribution function. This last cdf is an approximation of the so  called ``Devil's staircase'' (see Figure \ref{fig:Devilstaircase}). For the naive deconvolution estimator we find $h=0.011$ for $W_1$ and $h=0.018$ for $W_2$. For the $W_1$-isotone deconvolution estimator we find  $h = 0.002$ and $h = 0.01$ for the $W_2$-isotone estimator.  Note that these values are consistent with the fact that the
  bandwidth increases with the parameter $p$ of the Wasserstein metric, as shown by Theorem~\ref{theo:upperbounds}. On Figure \ref{fig: Cantor},  the $W_1$-isotone deconvolution estimator is able to ``see'' the first three  levels of the Cantor set and the three other deconvolution methods recover the first two levels. A kernel density estimator (with no deconvolution) only recovers the first level.

 \subsection{About the bandwidth choice} \label{sub:choiceh}

\begin{figure}
\begin{center}
	\includegraphics[scale=1]{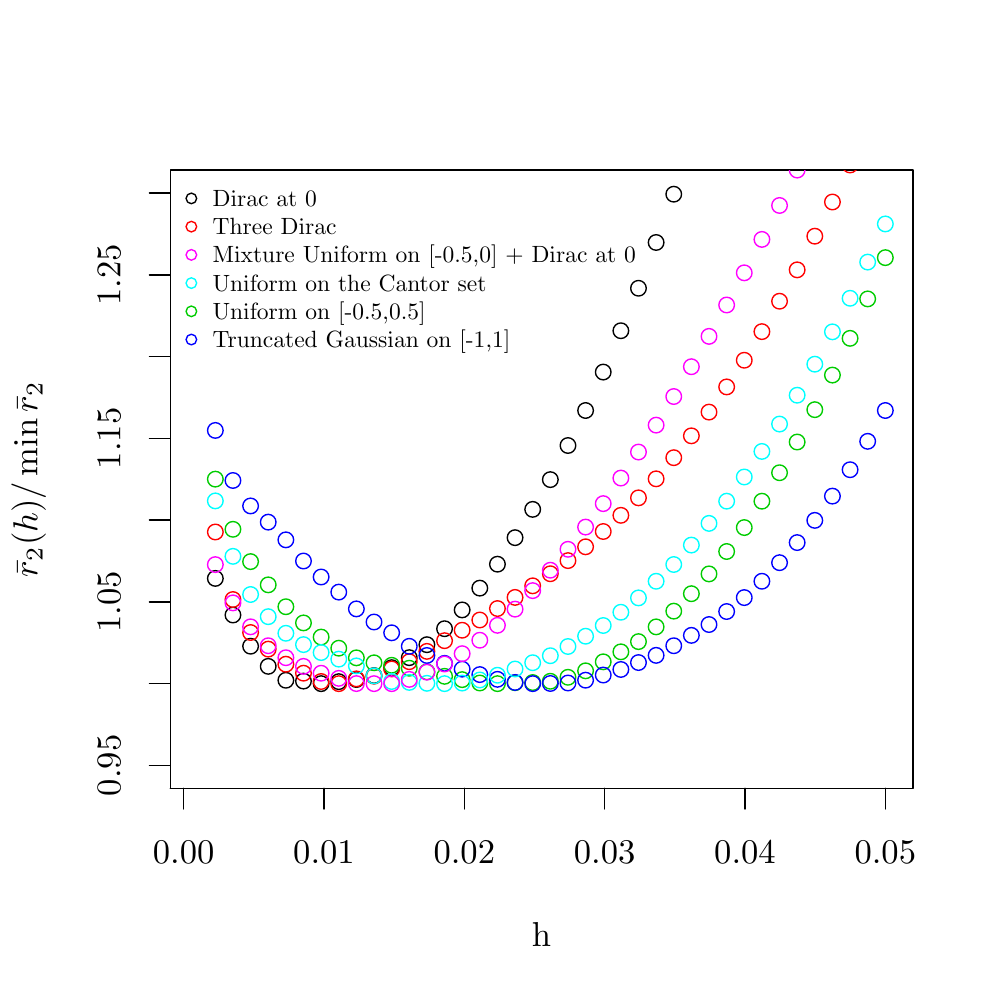}
\caption{Comparison of the locations of the minima of the $W_2$-risks for five distribution measures $\mu$. The noise distribution $\mu_\varepsilon$ is the symmetrized Gamma distribution with $\beta = 0.75$.
Each risk curve has been normalized by its minimum value for facilitating
the comparison.}
\label{fig: comparh}
\end{center}
\end{figure}

In practice, we need to choose a bandwidth $h$ for the deconvolution estimators.  As  was explained in  \cite{CCDM11} (see Remark 3 in this paper),  it seems that the influence of  the measure $\mu$ is weak. We now propose a simple experiment to check this principle. We choose for $\mu_\varepsilon$ the symmetrized gamma distribution with a shape parameter equal to $0.375$ ($\beta = 0.75$) and we simulate contaminated observations from the following various distributions:
\begin{itemize}
\item  Truncated standard Gaussian distribution on $[-1,1]$,
\item  Uniform distribution on $[-0.5,0.5]$,
\item Uniform distribution on the Cantor set,
\item Mixture of the Dirac distribution at 0 and the uniform distribution on $[-0.5,0]$,
\item Mixture of Dirac distributions at $-0.5$, $-0.2$ and $0.3$ with proportions $1/4$, $1/4$ and $1/2$,
\item Dirac distribution at $0$.
\end{itemize}
We focus here on the study of the $W_2$-isotone deconvolution estimator. Figure \ref{fig: comparh} compares the locations of the minimums of the five risk curves $h \mapsto \bar r  ^{\mbox{\tiny isot}}_{2,h}$ by averaging over 200 samples of 1000 contaminated observations. For this experiment, the sensitivity of the  minimum risk location to the distribution $\mu$  is not very large.

On another hand, from Figure \ref{fig:diracRates}, it seems
that the rates for the mixture Dirac Uniform are quite slow (in particular, they
are close to the minimax rates for $W_1$).

From these remarks, it seems that the bandwidth minimizing  the risk computed for the  mixture Dirac Uniform should be a reasonable choice for deconvolving other distributions.
Of course, this is in some sense a ``minimax choice'', and it will not give the appropriate
rate for measures which are  easier to estimate (for instance measures with smooth densities).


A bootstrap method in the spirit of \cite{delaigle2004bootstrap} may give a more satisfactory  answer to this problem. However, note that the use of the Wasserstein metric makes difficult the asymptotical analysis of the risk. This interesting problem is out of the scope of this paper, we intend to investigate it in a future work.

\begin{figure}
\begin{center}
	\includegraphics[scale=0.9]{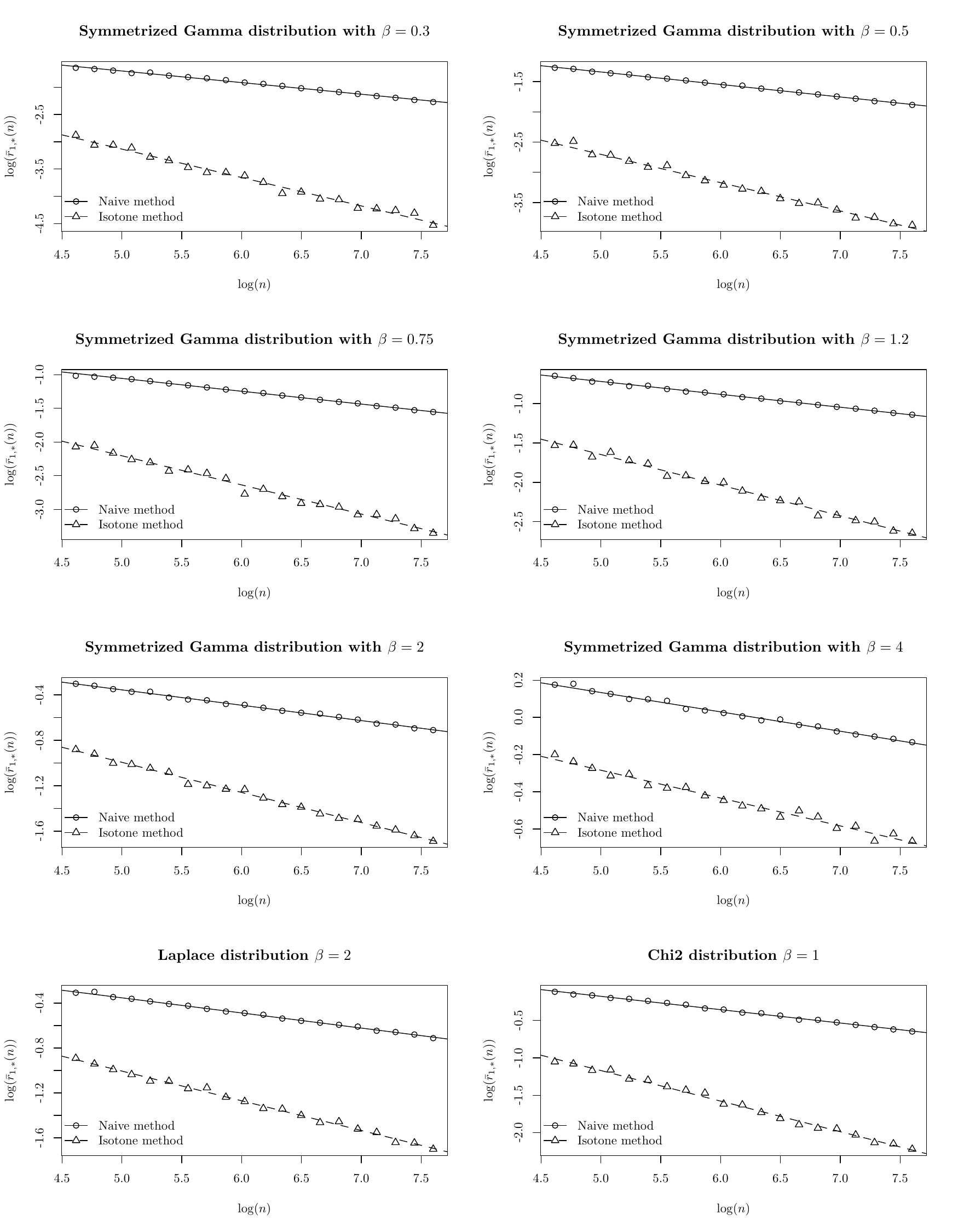}
\caption{Deconvolution of the Dirac distribution at zero observed with one of the noise distributions listed in Table~\ref{tab:distDirac}: log-log plots of the estimated $W_1$-risks for the naive method and the isotone method.}
\label{fig: LogLogW1}
\end{center}
\end{figure}

\begin{figure}
\begin{center}
	\includegraphics[scale=0.9]{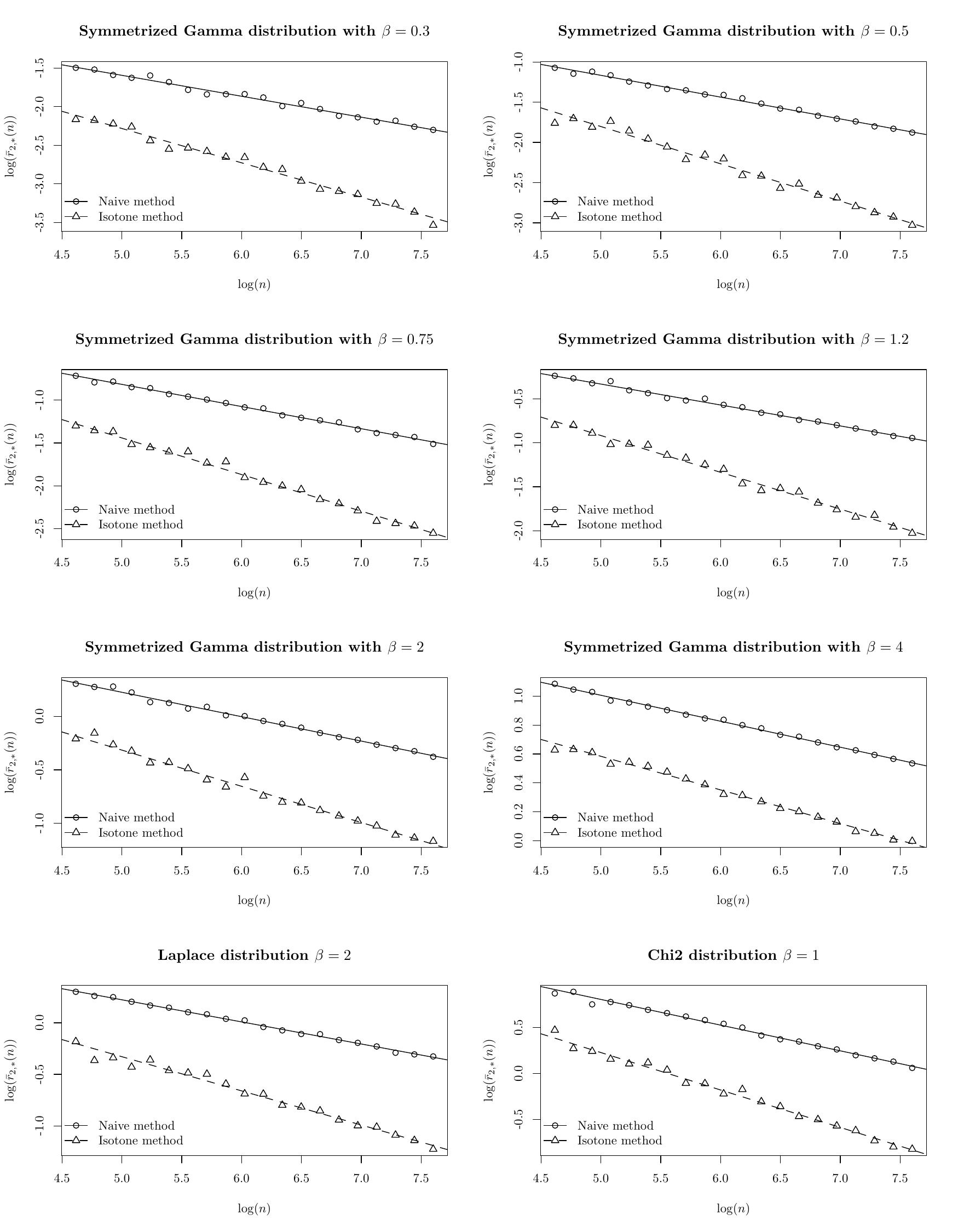}
\caption{Deconvolution of the Dirac distribution at zero observed with one of the noise distributions listed in Table~\ref{tab:distDirac}: log-log plots of the estimated $W_2$-risks for the naive method and the isotone method.}
\label{fig: LogLogW2}
\end{center}
\end{figure}

\section*{Acknowledgements}
The authors were supported by the ANR project TopData ANR-13-BS01-0008.
\bibliographystyle{plainnat}
\bibliography{BiblioWasserstein}

\begin{thebibliography}{26}
\providecommand{\natexlab}[1]{#1}
\providecommand{\url}[1]{\texttt{#1}}
\expandafter\ifx\csname urlstyle\endcsname\relax
  \providecommand{\doi}[1]{doi: #1}\else
  \providecommand{\doi}{doi: \begingroup \urlstyle{rm}\Url}\fi

\bibitem[Bobkov and Ledoux(2014)]{bobkov2014one}
S.~Bobkov and M.~Ledoux.
\newblock {One-dimensional empirical measures, order statistics and Kantorovich
  transport distances}.
\newblock \emph{Preprint}, 2014.

\bibitem[Butucea and Tsybakov(2008{\natexlab{a}})]{BT08a}
C.~Butucea and B~Tsybakov.
\newblock {Sharp optimality in density deconvolution with dominating bias. I}.
\newblock \emph{Theory Probab. Appl.}, 52:\penalty0 24--39, 2008{\natexlab{a}}.

\bibitem[Butucea and Tsybakov(2008{\natexlab{b}})]{BT08b}
C.~Butucea and B~Tsybakov.
\newblock {Sharp optimality in density deconvolution with dominating bias. II}.
\newblock \emph{Theory Probab. Appl.}, 52:\penalty0 237--249,
  2008{\natexlab{b}}.

\bibitem[Caillerie et~al.(2011)Caillerie, Chazal, Dedecker, and Michel]{CCDM11}
C.~Caillerie, F.~Chazal, J.~Dedecker, and B.~Michel.
\newblock {Deconvolution for the Wasserstein metric and geometric inference}.
\newblock \emph{Electron. J. Stat.}, 5:\penalty0 1394--1423, 2011.

\bibitem[Carlsson(2009)]{carlsson2009topology}
G.~Carlsson.
\newblock Topology and data.
\newblock \emph{Bull. Amer. Math. Soc.}, 46:\penalty0 255--308, 2009.

\bibitem[Carroll and Hall(1988)]{CH88}
R.J. Carroll and P.~Hall.
\newblock {Optimal rates of convergence for deconvolving a density}.
\newblock \emph{J. Amer. Statist. Assoc.}, 83:\penalty0 1184--1186, 1988.

\bibitem[Chazal et~al.(2011)Chazal, Cohen-Steiner, and M\'erigot]{CCSM11}
F.~Chazal, D.~Cohen-Steiner, and Q.~M\'erigot.
\newblock {Geometric inference for probability measures}.
\newblock \emph{Found. Comput. Math.}, 11:\penalty0 733--751, 2011.

\bibitem[Chazal et~al.(2014)Chazal, Fasy, Lecci, Michel, Rinaldo, and
  Wasserman]{chazal2014subsampling}
F.~Chazal, B.T. Fasy, F.~Lecci, B.~Michel, A.~Rinaldo, and L.~Wasserman.
\newblock Subsampling methods for persistent homology.
\newblock \emph{arXiv:1406.1901}, 2014.

\bibitem[Dattner et~al.(2011)Dattner, Goldenshluger, and Juditsky]{DGJ11}
I.~Dattner, A.~Goldenshluger, and A.~Juditsky.
\newblock {On deconvolution of distribution functions}.
\newblock \emph{Ann. Statist.}, 39:\penalty0 2477--2501, 2011.

\bibitem[Dedecker and Michel(2013)]{DM13}
J.~Dedecker and B.~Michel.
\newblock {Minimax rates of convergence for Wasserstein deconvolution with
  supersmooth errors in any dimension}.
\newblock \emph{J. Multivar. Anal.}, 122:\penalty0 278--291, 2013.

\bibitem[del Barrio et~al.(1999)del Barrio, Gin\'e, and Matr\'an]{BGM99}
E.~del Barrio, E.~Gin\'e, and C.~Matr\'an.
\newblock {The central limit theorem for the Wasserstein distance between the
  empirical and the true distributions}.
\newblock \emph{Ann. Probab.}, 27:\penalty0 1009--1971, 1999.

\bibitem[del Barrio et~al.(2005)del Barrio, Gin\'e, and Utzet]{BGU05}
E.~del Barrio, E.~Gin\'e, and F.~Utzet.
\newblock {Asymptotics for ${\mathbb L}_2$ functionals of the empirical
  quantile process, with applications to tests of fit based on weighted
  Wasserstein distances}.
\newblock \emph{Bernoulli}, 11:\penalty0 131--189, 2005.

\bibitem[Delaigle and Gijbels(2004)]{delaigle2004bootstrap}
A.~Delaigle and I.~Gijbels.
\newblock Bootstrap bandwidth selection in kernel density estimation from a
  contaminated sample.
\newblock \emph{Ann. I. Stat. Math.}, 56\penalty0 (1):\penalty0 19--47, 2004.

\bibitem[Dereich et~al.(2013)Dereich, Scheutzow, and Schottstedt]{DSS13}
S.~Dereich, M.~Scheutzow, and R.~Schottstedt.
\newblock {Constructive quantization: Approximation by empirical measures}.
\newblock \emph{Ann. Inst. H. Poincar\'e Probab. Statist.}, 49:\penalty0
  1183--1203, 2013.

\bibitem[{\`E}bralidze(1971)]{E71}
{\v{S}}.S. {\`E}bralidze.
\newblock Inequalities for the probabilities of large deviations in terms of
  pseudomoments.
\newblock \emph{Teor. Verojatnost. i Primenen.}, 16:\penalty0 760--765, 1971.

\bibitem[Fan(1991{\natexlab{a}})]{Fan91}
J.~Fan.
\newblock {On the optimal rates of convergence for nonparametric deconvolution
  problems}.
\newblock \emph{Ann. Stat.}, 19:\penalty0 1257--1272, 1991{\natexlab{a}}.

\bibitem[Fan(1991{\natexlab{b}})]{Fan91a}
J.~Fan.
\newblock {Global behavior of deconvolution kernel estimates}.
\newblock \emph{Statist. Sinica}, 2:\penalty0 541--551, 1991{\natexlab{b}}.

\bibitem[Fan(1993)]{Fan93}
J.~Fan.
\newblock {Adaptively local one-dimensional subproblems with application to a
  deconvolution problem}.
\newblock \emph{Ann. Stat.}, 21:\penalty0 600--610, 1993.

\bibitem[Fournier and Guillin(2013)]{FG}
N.~Fournier and A.~Guillin.
\newblock On the rate of convergence in wasserstein distance of the empirical
  measure.
\newblock \emph{Preprint}, 2013.

\bibitem[Guibas et~al.(2013)Guibas, Morozov, and
  M{\'e}rigot]{guibas2013witnessed}
L.~Guibas, D.~Morozov, and Q.~M{\'e}rigot.
\newblock Witnessed k-distance.
\newblock \emph{Discrete Comput. Geom.}, 49:\penalty0 22--45, 2013.

\bibitem[Hall and Lahiri(2008)]{HL}
P.~Hall and S.N. Lahiri.
\newblock {Estimation of distributions, moments and quantiles in deconvolution
  problems}.
\newblock \emph{Ann. Statist}, 36:\penalty0 2110--2134, 2008.

\bibitem[Mair et~al.(2009)Mair, Hornik, and de~Leeuw]{mair2009isotone}
P.~Mair, K.~Hornik, and J.~de~Leeuw.
\newblock Isotone optimization in $\text{R}$: pool-adjacent-violators algorithm
  ($\text{PAVA}$) and active set methods.
\newblock \emph{J. Stat. Softw.}, 32\penalty0 (5):\penalty0 1--24, 2009.

\bibitem[Meister(2009)]{Mei09}
A.~Meister.
\newblock \emph{{Deconvolution Problems in Nonparametric Statistics}}.
\newblock Lecture Notes in Statistics. Springer, 2009.

\bibitem[Rachev and R\"uschendorf(1998)]{RR98}
S.T. Rachev and L.~R\"uschendorf.
\newblock \emph{{Mass transportation problems}}, volume~II of \emph{Probability
  and its Applications}.
\newblock Springer-Verlag, 1998.

\bibitem[van~der Vaart and Wellner(1996)]{vdVW96}
A.W. van~der Vaart and J.A. Wellner.
\newblock \emph{{Weak Convergence and Empirical Processes}}.
\newblock Springer series in Statistics. Springer, 1996.

\bibitem[Villani(2008)]{Vil08}
C.~Villani.
\newblock \emph{{Optimal Transport: Old and New}}.
\newblock Grundlehren Der Mathematischen Wissenschaften. Springer-Verlag, 2008.

\end{thebibliography}

\end{document}